\documentclass[3p]{elsarticle}

\usepackage{amssymb, amsmath, amsthm}
\newcommand{\jmp}[1]{[\![#1]\!]}              
\newcommand{\mvl}[1]{\{\!\!\{#1\}\!\!\}}      
\newtheorem{lemma}{Lemma}
\newtheorem{theorem}{Theorem}
\newtheorem{remark}{Remark}

\journal{Journal of Computational and Applied Mathematics}

\begin{document}
 
\begin{frontmatter}
\title{High-order explicit local time-stepping methods for damped wave equations\tnoteref{t1}}
\tnotetext[t1]{This work was supported by the Swiss National Science Foundation.}

\author{Marcus J.~Grote\corref{cor1}} \cortext[cor1]{Corresponding author: Tel. +41612673996, Fax +41612672695} \ead{Marcus.Grote@unibas.ch}
\author{Teodora Mitkova} \ead{Teodora.Mitkova@unibas.ch}

\address{Institute of Mathematics, University of Basel, Rheinsprung 21, 4051 Basel, Switzerland}

\begin{abstract}
Locally refined meshes impose severe stability constraints on explicit time-stepping methods for the numerical simulation of time dependent wave phenomena. 
Local time-stepping methods overcome that bottleneck by using smaller time-steps precisely where the smallest elements in the mesh are located.
Starting from classical Adams-Bashforth multi-step methods, local time-stepping methods of arbitrarily high order of accuracy are derived for damped wave equations. 
When combined with a finite element discretization in space with an essentially diagonal mass matrix, the resulting time-marching schemes are fully explicit 
and thus inherently parallel. Numerical experiments with continuous and discontinuous Galerkin finite element discretizations corroborate the expected rates of convergence and illustrate 
the usefulness of these local time-stepping methods.
\end{abstract}

\end{frontmatter}

\section{Introduction}
Efficient numerical methods for the simulation of damped wave phenomena are of fundamental importance in acoustic, electromagnetic or seismic wave propagation.
In the presence of complex geometry, such as cracks, sharp corners or irregular material interfaces, standard finite difference methods generally become ineffective 
and cumbersome. In contrast, finite element methods (FEMs) easily handle unstructured meshes and local refinement. Moreover, their extension to high order
is straightforward, a key feature to keep numerical dispersion minimal.

The finite element Galerkin approximation of hyperbolic problems typically leads to a system of ordinary differential equations. However, if explicit time-stepping is 
subsequently employed, the mass matrix arising from the spatial discretization by standard conforming finite elements must be inverted at each time-step: a major drawback 
in terms of efficiency. To overcome that difficulty, various ``mass lumping'' techniques have been proposed, which effectively replace the mass matrix by a diagonal 
approximation. While straightforward for piecewise linear elements \cite{Ciar78,Hug00}, mass lumping techniques require special quadrature rules at higher order
to preserve the accuracy and guarantee numerical stability \cite{CJRT01,GT06}.

Discontinuous Galerkin (DG) methods offer an attractive and increasingly popular alternative for the spatial discretization of time-dependent hyperbolic problems 
\cite{CS89, RW03, AMM06, GSS06,GSS07, HW08}. Not only do they accommodate elements of various types and shapes, irregular non-matching grids, and even locally varying polynomial order, and hence 
offer greater flexibility in the mesh design. They also lead to a block-diagonal mass matrix, with block size equal to the number of degrees of freedom per element; 
in fact, for a judicious choice of (locally orthogonal) shape functions, the mass matrix is truly diagonal. Thus, when a spatial DG discretization is combined with 
explicit time integration, the resulting time-marching scheme will be truly explicit and inherently parallel.

In the presence of complex geometry, adaptivity and mesh refinement are certainly key for the efficient numerical simulation of wave phenomena. However, locally refined 
meshes impose severe stability constraints on explicit time-marching schemes, where the maximal time-step allowed by the CFL condition is dictated by the smallest elements 
in the mesh. When mesh refinement is restricted to a small region, the use of implicit methods, or a very small time-step in the entire computational domain, are very 
high a price to pay. To overcome this overly restricitve stability constraint, various local time-stepping (LTS) schemes \cite{CFJ03, CFJ06, DFFL10} were developed, 
which use either implicit time-stepping or explicit smaller time-steps, but only where the smallest elements in the mesh are located.

Since DG methods are inherently local, they are particularly well-suited for the development of explicit local time-stepping schemes \cite{HW08}. By combining the sympletic 
St\"ormer-Verlet method with a DG discretization, Piperno derived a symplectic LTS scheme for Maxwell's equations in a non-conducting medium \cite{Pip06}, which is explicit 
and second-order accurate. In \cite{MPFC08}, Montseny et al. combined a similar recursive integrator with discontinuous hexahedral elements. Starting from the so-called 
arbitrary high-order derivatives (ADER) DG approach, alternative explicit LTS methods for Maxwell's equations \cite{TDMS09} and for elastic wave equations \cite{DKT07} were 
proposed. In \cite{EJ10}, the LTS approach from Collino et al. \cite{CFJ03, CFJ06} was combined with a DG-FE discretization for the numerical solution of symmetric first-order hyperbolic systems. 
Based on energy conservation, that LTS approach is second-order and explicit inside the coarse and the fine mesh; at the interface, however, it nonetheless requires at 
every time-step the solution of a linear system. More recently, Constantinescu and Sandu devised multirate explicit methods for hyperbolic conservation laws, 
which are based on both Runge-Kutta and Adams-Bashforth schemes combined with a finite volume discretization \cite{CS07, CS09}. Again these multirate schemes are limited to 
second-order accuracy.

Starting from the standard leap-frog method, Diaz and Grote proposed energy conserving fully explicit LTS integrators of arbitrarily high accuracy for the classical wave 
equation \cite{DG09}; that approach was extended to Maxwell's equations in \cite{GM10} for non-conductive media. By blending the leap-frog and the Crank-Nicolson methods, 
a second-order LTS scheme was also derived there for (damped) electromagnetic waves in conducting media, yet  this approach cannot be readily extended beyond order 
two.

To achieve arbitrarily high accuracy in the presence of dissipation, while remaining fully explicit, we shall derive here explicit LTS methods for damped wave equations based 
on Adams-Bashforth (AB) multi-step schemes. They can also be interpreted as particular
approximations of exponential-Adams multistep methods \cite{HO11}. The rest of the paper is organized as follows. In section 2, we first recall the standard continuous and the symmetric 
interior penalty (IP) DG finite element discretizations of the second-order damped wave equation; there, we also rewrite the damped wave equation as a first-order hyperbolic system 
and recall its nodal DG formulation. Starting from the Adams-Bashforth multi-step schemes, we then derive in Section 3 a LTS approach of arbitrarily high accuracy. 
When combined with a finite element discretization in space with a diagonal mass matrix, the resulting time-marching schemes remain fully explicit. Finally in Section 4, 
we present numerical experiments in one and two space dimensions, which validate the theory and underpin both the stability properties and the usefulness of these 
Adams-Bashforth LTS schemes. 

\section{Finite element discretizations for the damped wave equation}
We consider the scalar damped wave equation
\begin{equation} \label{model_eq_1} \begin{split}
u_{tt} + \sigma u_t - \nabla \cdot (c^2 \nabla u) & = f \quad \ \mbox{in } \ \Omega \times (0, T) \,,\\
u(\cdot, t) & = 0 \quad \ \mbox{on } \ \partial \Omega \times (0, T)\,, \\
u(\cdot, 0) = u_0\,, \  u_t(\cdot, 0) & = v_0 \quad \mbox{in } \Omega \,, \\
\end{split} \end{equation}
where $\Omega$ is a bounded domain in ${\mathbb R}^d$, $d = 1, 2, 3$. Here, $f\in L^2(0,T; L^2(\Omega))$ is a (known) source term,
while $u_0\in H^1_0(\Omega)$ and $v_0\in L^2(\Omega)$ are prescribed initial conditions. At the boundary, $\partial \Omega$, we impose 
a homogeneous Dirichlet boundary condition, for simplicity. The damping coefficient, $\sigma = \sigma(x)$, is assumed non-negative 
($\sigma \geq 0$) whereas the speed of propagation, $c = c(x)$, is piecewise smooth and strictly positive ($c(x) \geq c_0  > 0$).

We shall now discretize (\ref{model_eq_1}) in space by using any one of the following three distinct FE discretizations: 
continuous ($H^1$-conforming) finite elements with mass lumping, a symmetric IP-DG 
discretization, or a nodal DG method. Thus, we consider shape-regular 
meshes ${\mathcal T}_h$ that partition the domain $\Omega$ into disjoint elements $K$, such that 
$\overline \Omega = \cup_{K \in {\mathcal T}_h} \overline K$. The elements are triangles or quadrilaterals in two space dimensions, 
and tetrahedra or hexahedra in three dimensions, respectively. The diameter of element $K$ is denoted by $h_K$ and the mesh size, 
$h$, is given by $h = \max_{K \in {\mathcal T}_h} h_K$.

\subsection{Continuous Galerkin formulation}
The continuous ($H^1$-conforming) Galerkin formulation of (\ref{model_eq_1}) starts from its weak formulation: 
find 
$u \in [0, T] \to H^1_0(\Omega)$ such that
\begin{equation} \label{weak_eq} \begin{split}
(u_{tt}, \varphi) + (\sigma u_t, \varphi) + (c \nabla u, c \nabla \varphi) & = (f, \varphi) 
\qquad \forall \ \varphi \in H_0^1(\Omega)\,, \quad t \in (0, T) \,, \\
u(\cdot, 0) = u_0 \,, \ u_t(\cdot, 0) & = v_0 \,,
\end{split} \end{equation}
where $(\cdot, \cdot)$ denotes the standard $L^2$ inner product over $\Omega$. It is well-known that (\ref{weak_eq}) is well-posed and has a unique 
solution \cite{LM72}.

For a given partition ${\mathcal T}_h$ of $\Omega$, assumed polygonal for simplicity, and an approximation order $\ell \geq 1$, we shall approximate the
solution $u(\cdot, t)$ of (\ref{weak_eq}) in the finite element space
$$
V^h := \left \{ \varphi \in H_0^1(\Omega) \ : \ \varphi|_{K} \in {\mathcal S}^{\ell}(K) \ \
\forall \ K \in {\mathcal T}_h \right \}\,,
$$
where ${\mathcal S}^{\ell}(K)$ is the space ${\mathcal P}^{\ell}(K)$ (for triangles or tetrahedra) or
${\mathcal Q}^{\ell}(K)$ (for quadrilaterals or hexahedra). Here, we consider the following semi-discrete Galerkin approximation of
(\ref{weak_eq}): find $u^h : [0, T] \to V^h$ such that
\begin{equation} \label{fe_eq} \begin{split}
(u_{tt}^h, \varphi) + (\sigma u_{t}^h, \varphi) + (c \nabla u^h, c \nabla \varphi) & = (f, \varphi) 
\qquad \forall \, \varphi \in V^h\,, \quad t \in (0, T) \,, \\
u^h(\cdot, 0) = {\Pi}_h u_0 \,, \ u^h_t(\cdot, 0) &= {\Pi}_h v_0 \,.
\end{split} \end{equation}
Here, ${\Pi}_h$ denotes the $L^2$-projection onto $V^h$.

The semi-discrete formulation (\ref{fe_eq}) is equivalent to the second-order system of ordinary differential equations
\begin{equation*} \begin{split}
{\mathbf M} \, \frac{d^2 {\mathbf U}}{d t^2}(t) + {\mathbf M}_{\sigma} \, \frac{d {\mathbf U}}{d t}(t) +
{\mathbf K} \, {\mathbf U}(t) & = {\mathbf F}(t)\,, \qquad t \in (0, T)\,, \\
{\mathbf M} \, {\mathbf U}(0) = u_0^h \,, \qquad {\mathbf M} \, \frac{d \mathbf U}{d t}(0) & = v_0^h \,.
\end{split} \end{equation*}
Here, ${\mathbf U}$ denotes the vector whose components are the coefficients of $u^h$ with respect to the finite element basis of
$V_h$, $\mathbf M$ the mass matrix, $\mathbf K$ the stiffness matrix, whereas  ${\mathbf M}_{\sigma}$ denotes
the mass matrix with weight $\sigma$. The matrix ${\mathbf M}$ is sparse, symmetric and positive
definite, whereas the matrices $\mathbf K$ and ${\mathbf M}_{\sigma}$ are sparse, symmetric and, in general, only positive semi-definite. In fact, $\mathbf K$ is positive definite, unless Neumann boundary conditions would be imposed
in (\ref{model_eq_1}) instead. Since we shall never need to invert $\mathbf K$, our derivation also applies to
the semi-definite case with purely Neumann boundary conditions.

Usually, the mass matrix ${\mathbf M}$ is not diagonal, yet needs to be inverted at every time-step of any explicit time integration scheme. 
To overcome this diffculty, various mass lumping techniques have been developed \cite{CJRT01, Coh02,CJT93,CJT94}, which essentially replace $\mathbf M$ 
with a diagonal approximation by computing the required integrals over each element $K$ with judicious quadrature rules that do not effect the 
spatial accuracy \cite{BD76}.

\subsection{Interior penalty discontinuous Galerkin formulation}
Following \cite{GSS06} we briefly recall the symmetric interior penalty (IP) DG formulation of (\ref{model_eq_1}). For simplicity,
we assume in this section that the elements are triangles or quadrilaterals in two space dimensions and tetrahedra or hexahedra
in three dimensions, respectively. Generally, we allow for irregular ($k$-irregular) meshes with hanging nodes \cite{BPW09}. We denote 
by ${\mathcal E}^{\mathcal I}_h$ the set of all interior edges of ${\mathcal T}_h$, by ${\mathcal E}^{\mathcal B}_h$ the set of all 
boundary edges of ${\mathcal T}_h$, and set ${\mathcal E}_h = {\mathcal E}^{\mathcal I}_h \cup {\mathcal E}^{\mathcal B}_h$. Here, we 
generically refer to any element of ${\mathcal E}_h$ as an ``edge'', both in two and three space dimensions.

For a piecewise smooth function $\varphi$, we introduce the following trace operators. Let $e \in {\mathcal E}^{\mathcal I}_h$ be an
interior edge shared by two elements $K^+$ and $K^-$ with unit outward normal vectors ${\mathbf n}^{\pm}$, respectively. Denoting
by $v^{\pm}$ the trace of $v$ on $\partial K^\pm$ taken from within $K^\pm$, we define the jump and the average on $e$ by
$$
\jmp{\varphi} := \varphi^+ {\mathbf n}^+ + \varphi^-{\mathbf n}^-  \,, \qquad
\mvl{\varphi}:= (\varphi^+ + \varphi^-)/2\,.
$$
On every boundary edge $e \in {\mathcal E}^{\mathcal B}_h$, we set $\jmp{\varphi} := \varphi {\mathbf n}$ and 
$\mvl{\varphi} := \varphi$. Here, $\mathbf{n}$ is the outward unit normal to the domain boundary $\partial \Omega$.

For a piecewise smooth vector-valued function $\boldsymbol{\psi}$, we analogously define the average across interior faces by
$\mvl{\boldsymbol{\psi}}:=(\boldsymbol{\psi}^+ + \boldsymbol{\psi}^-)/2$, and on boundary faces we set 
$\mvl{\boldsymbol{\psi}}:=\boldsymbol{\psi}$. The jump of a
vector-valued function will not be used.  For a vector-valued function $\boldsymbol{\psi}$ with continuous normal components across a
face $e \in {\mathcal E}_h$, the trace identity
\begin{equation*}
\varphi^+\left( \mathbf{n}^{+} \cdot \boldsymbol{\psi}^+ \right) + \varphi^-\left( \mathbf{n}^{-} \cdot \boldsymbol{\psi}^- \right)
 = \jmp{\varphi} \cdot \mvl{\boldsymbol{\psi}} \quad \mbox{on } e\,,
\end{equation*}
immediately follows from the above definitions.

For a given partition ${\cal T}_h$ of $\Omega$ and an approximation order $\ell \geq 1$, we wish to approximate the solution 
$u(t, \cdot)$ of (\ref{model_eq_1}) in the finite element space
\begin{equation*}
V^h :=  \left \{ \varphi \in L^2(\Omega): \, \varphi|_K\in {\cal S}^{\ell}(K) ~~ \forall K\in {\cal T}_h \right \} \,,
\end{equation*}
where ${\cal S}^{\ell}(K)$ is the space ${\cal P}^{\ell}(K)$ of polynomials of total degree at most $\ell$ on $K$ if $K$ is a
triangle or a tetrahedra, or the space ${\cal Q}^{\ell}(K)$ of polynomials of degree at most $\ell$ in each variable on $K$ if $K$
is a quadrilateral or a hexahedral. 
Thus, we consider the following (semidiscrete) DG approximation of (\ref{model_eq_1}): find
$u^h : [0, T] \to V^h$ such that
\begin{equation} \label{dg_eq} \begin{split}
(u_{tt}^h, \varphi) + (\sigma u_{t}^h, \varphi) + a_h(u^h, \varphi) & = (f,\varphi) 
\qquad \forall \, \varphi \in V^h\,, \quad t \in (0, T) \,, \\
u^h(\cdot, 0) = {\Pi}_h u_0 \,, \ u^h_t(\cdot, 0) &= {\Pi}_h v_0 \,.
\end{split} \end{equation}
Here, ${\Pi}_h$ again denotes the $L^2$-projection onto $V^h$ whereas the DG bilinear form $a_h(\cdot, \cdot)$, defined on $V^h \times V^h$, is given by
\begin{equation} \label{dg_form} \begin{split}
a_{h}(u, \varphi) :=  & \sum_{K \in {\mathcal T}_h} \int_K c^2 \, \nabla u \cdot \nabla \varphi \,dx - 
\sum_{e \in {\mathcal E}_h} \int_{e} \jmp{u} \cdot \mvl{c^2 \, \nabla \varphi} \, dA \\
& -\sum_{e \in{\mathcal E}_h} \int_{e} \jmp{\varphi} \cdot \mvl{c^2 \, \nabla u} \, dA + 
\sum_{e \in {\mathcal E}_h} \int_{e}\, {\tt a} \, \jmp{u} \cdot \jmp{\varphi} \,dA \, .
\end{split} \end{equation}
The last three terms in (\ref{dg_form}) correspond to jump and flux terms at element boundaries; they vanish when 
$u, \varphi, \in H^1_0(\Omega)\cap H^{1+m}(\Omega)$ for $m > \frac{1}{2}$. Hence, the above semi-discrete DG formulation (\ref{dg_eq}) 
is consistent with the original continuous problem~(\ref{weak_eq}).

In (\ref{dg_form}) the function ${\tt a}$ penalizes the jumps of $u$ and $v$ over the faces of ${\mathcal T}_h$. To define it, we first
introduce the functions $\tt h$ and $\tt c$ by
$$
{\tt h} |_e = \left \{ \begin{array}{ll}
\min \{ h_{K^+}, h_{K^-} \}, & e \in {\mathcal E}_h^{\mathcal I}\,,\\[5pt] h_K, & e \in {\mathcal E}_h^{\mathcal B}\,,
\end{array} \right. \
{\tt c} |_e(x) = \left \{ \begin{array}{ll}
\max \{ c|_{K^+}(x), c|_{K^-}(x) \}, & e \in {\mathcal E}_h^{\mathcal I} \,,\\[5pt] c|_K(x), & e \in {\mathcal E}_h^{\mathcal B} \,.
\end{array} \right.
$$
Then, on each $e \in {\mathcal E}_h$, we set
\begin{equation} \label{param_alpha}
{\tt a} |_e := \alpha \, {\tt c}^2 {\tt h}^{-1}\,,
\end{equation}
where $\alpha$ is a positive parameter independent of the local mesh sizes and the coefficient $c$. There exists a threshold value
$\alpha_{min} > 0$, which depends only on the shape regularity of the mesh  and the approximation order $\ell$ such that for  
$\alpha \geq \alpha_{min}$ the DG bilinear form $a_h$ is coercive and, hence, the discretization stable \cite{ABCM01, AD10}. Throughout the 
rest of the paper we shall assume that $\alpha \geq \alpha_{min}$ so that the semi-discrete problem (\ref{dg_eq}) has 
a unique solution which converges with optimal order \cite{GSS06,GSS07,GSS08, GS09}. In \cite{GSS06,GS09}, a detailed convergence analysis 
and numerical study of the IP-DG method for (\ref{dg_form}) with $\sigma = 0$ was presented. In particular, optimal a-priori estimates 
in a DG-energy norm and the $L^2$-norm were derived. This theory immediately generalizes to the case $\sigma \geq 0$. For sufficiently 
smooth solutions, the IP-DG method (\ref{dg_form}) thus yields the optimal $L^2$-error estimate of order ${\mathcal O}(h^{\ell + 1})$.

The semi-discrete IP-DG formulation (\ref{dg_eq}) is equivalent to the second-order system of ordinary differential equations
\begin{equation*} \begin{split}
{\mathbf M} \, \frac{d^2 {\mathbf U}}{d t^2}(t) + {\mathbf M}_{\sigma} \, \frac{d {\mathbf U}}{d t}(t) +
{\mathbf K} \, {\mathbf U}(t) & = {\mathbf F}(t)\,, \qquad t \in (0, T)\,, \\
{\mathbf M} \, {\mathbf U}(0) = u_0^h \,, \qquad {\mathbf M} \, \frac{d \mathbf U}{d t}(0) & = v_0^h \,.
\end{split} \end{equation*}
Again, the mass matrix ${\mathbf M}$ is sparse, symmetric and positive definite. Yet because individual elements decouple, 
${\mathbf M}$ (and ${\mathbf M}_{\sigma}$) is block-diagonal with block size equal to the number of degrees of freedom per element. 
Thus, ${\mathbf M}$ can be inverted at very low computational cost. In fact, for a judicious choice of (locally orthogonal) shape functions, 
${\mathbf M}$ is truly diagonal.

\subsection{Nodal discontinuous Galerkin formulation}
Finally, we briefly recall the nodal discontinuous Galerkin formulation from \cite{HW08} for the spatial discretization of (\ref{model_eq_1}) 
rewritten as a first-order system. To do so, we first let $v := u_t$, ${\mathbf w}:=-\nabla u$, and thus we rewrite (\ref{model_eq_1}) as the 
first-order hyperbolic system:
\begin{equation} \label{model_eq_tmp} \begin{split}
v_t + \sigma v + \nabla \cdot (c^2 {\mathbf w}) & = f \quad \quad \quad \mbox{in } \ \Omega \times (0, T) \,,\\
{\mathbf w}_t + \nabla v & = {\mathbf 0} \quad \quad \quad \mbox{in } \ \Omega \times (0, T) \,,\\
v(\cdot, t) &= 0 \quad \quad \quad  \mbox{on } \ \partial \Omega \times (0, T)\,, \\
v(\cdot, 0) = v_0\,, \ {\mathbf w}(\cdot, 0) & = - \nabla u_0  \quad \mbox{in } \Omega \,,
\end{split} \end{equation}
or in more compact notation as 
\begin{equation} \label{model_eq_2}
{\mathbf q}_t + {\mathbf \Sigma} \, {\mathbf q} + \nabla \cdot {\mathcal F}({\mathbf q}) = {\mathbf S} \,,
\end{equation}
with ${\mathbf q} = (v, {\mathbf w})^T$. Following \cite{HW08}, we now consider the following nodal DG formulation of (\ref{model_eq_2}): 
find ${\mathbf q}^h : [0, T] \to {\mathbf V}^h$ such that
\begin{equation} \label{dg1_pr}
({\mathbf q}_{t}^h, {\boldsymbol \psi}) + ({\mathbf \Sigma} \, {\mathbf q}^h, {\boldsymbol \psi}) + 
{\widetilde a}_h({\mathbf q}^h, {\boldsymbol \psi}) = ({\mathbf S}, {\boldsymbol \psi}) 
\qquad \forall \, {\boldsymbol \psi} \in {\mathbf V}^h\,, \quad t \in (0, T) \,.
\end{equation}
Here ${\mathbf V}^h$ denotes the finite element space
\begin{equation*}
{\mathbf  V}^h :=  \left \{ {\boldsymbol \psi} \in L^2(\Omega)^{d+1}: \, {\boldsymbol \psi}|_K\in {\cal S}^{\ell}(K)^{d+1} ~~ \forall K\in {\cal T}_h \right \}
\end{equation*}
for a given partition ${\cal T}_h$ of $\Omega$ and an approximation order $\ell \geq 1$. The nodal-DG bilinear form ${\widetilde a}_h(\cdot , \cdot)$ is defined on 
${\mathbf V}^h \times {\mathbf V}^h$ as
$$
{\widetilde a}_h({\mathbf q}, {\boldsymbol \psi}) := 
\sum_{K \in {\cal T}_h} \int_K \left ( \nabla \cdot {\mathcal F}({\mathbf q}) \right ) \cdot {\boldsymbol \psi} \, dx - 
\sum_{e \in {\mathcal E}_h} \int_e \left ( {\mathbf n} \cdot {\mathcal F}({\mathbf q}) - 
({\mathbf n} \cdot {\mathcal F}({\mathbf q}))^* \right ) \cdot {\boldsymbol \psi} \, dA\,,
$$
where $({\mathbf n} \cdot {\mathcal F}({\mathbf q}))^*$ is a suitably chosen numerical flux in the unit normal direction $\mathbf n$. The 
semi-discrete problem (\ref{dg1_pr}) has a unique solution, which converges with optimal order in the $L^2$-norm \cite{HW08}.

The semi-discrete nodal DG formulation (\ref{dg1_pr}) is equivalent to the first-order system of ordinary differential equations
\begin{equation} \label{sd_pr2}
{\mathbf M} \frac{d {\mathbf Q}}{dt}(t) + {\mathbf M}_\sigma \, {\mathbf Q}(t) + {\mathbf C} \, {\mathbf Q}(t) = {\mathbf F}(t) \,, 
\qquad t \in (0, T)\,.
\end{equation}
Here ${\mathbf Q}$ denotes the vector whose components are the coefficients of ${\mathbf q}^h$ with respect to the finite element basis of 
${\mathbf V}^h$ and $\mathbf C$ the DG stiffness matrix. Because the individual elements decouple, the mass matrices ${\mathbf M}$ and 
${\mathbf M}_{\sigma}$ are sparse, symmetric, positive semi-definite and block-diagonal; moreover, ${\mathbf M}$ is positive definite and can be inverted at 
very low computational cost.

\section{High-order explicit local time-stepping}
The $H^1$-conforming and the IP-DG finite element discretizations of (\ref{model_eq_1}) presented in Section 2 lead to the second-order system of 
differential equations
\begin{equation} \label{sdisc_eq1}
{\mathbf M} \, \frac{d^2 {\mathbf U}}{d t^2}(t) + {\mathbf M}_{\sigma} \, \frac{d {\mathbf U}}{d t}(t) + {\mathbf K} \,
{\mathbf U}(t) = {\mathbf 0}\,, \qquad t \in (0, T)\,,
\end{equation}
whereas the nodal DG discretization leads to the first-order system of differential equations
\begin{equation} \label{sdisc_eq4}
{\mathbf M} \frac{d {\mathbf Q}}{dt}(t) + {\mathbf M}_\sigma \, {\mathbf Q}(t) + {\mathbf C} \, {\mathbf Q}(t) = {\mathbf 0} \,, 
\qquad t \in (0, T)\,.
\end{equation}
In both (\ref{sdisc_eq1}) and (\ref{sdisc_eq4}) the mass matrix ${\mathbf M}$ is symmetric, positive definite and essentially diagonal; thus, ${\mathbf M}^{-1}$ or 
${\mathbf M}^{-\frac{1}{2}}$ can be computed explicitly at a negligible cost; for simplicity, we restrict ourselves here to the homogeneous case, 
i.e. ${\mathbf F}(t) = \mathbf{0}$.

If we multiply (\ref{sdisc_eq1}) by ${\mathbf M}^{-\frac{1}{2}}$, we obtain
\begin{equation} \label{sdisc_eq2}
\frac{d^2 {\mathbf z}}{d t^2}(t) + {\mathbf D} \, \frac{d {\mathbf z}}{d t}(t) + {\mathbf A} \, {\mathbf z}(t) = {\mathbf 0} \,,
\end{equation}
with ${\mathbf z}(t) = {\mathbf M}^{\frac{1}{2}} {\mathbf U}(t)$, ${\mathbf D} = {\mathbf M}^{-\frac{1}{2}} {\mathbf M}_{\sigma}
{\mathbf M}^{-\frac{1}{2}}$ and ${\mathbf A} = {\mathbf M}^{-\frac{1}{2}} {\mathbf K} {\mathbf M}^{-\frac{1}{2}}$. Note that
$\mathbf A$ is also sparse and symmetric positive semi-definite. Thus, we can rewrite (\ref{sdisc_eq2}) as a 
first-order problem of the form
\begin{equation} \label{sdisc_eq3}
\frac{d {\mathbf y}}{d t}(t) = {\mathbf B} {\mathbf y}(t)\,,
\end{equation} with
$$
{\mathbf y}(t) = \left ( {\mathbf z}(t), \frac{d {\mathbf z}}{d t}(t) \right )^T\,, \qquad {\mathbf B} = \left ( \begin{array}{rr}
{\mathbf 0} & {\mathbf I} \\ -{\mathbf A} & -{\mathbf D} \end{array} \right ) \,.
$$
Similarly, we can also rewrite (\ref{sdisc_eq4}) as in the form (\ref{sdisc_eq3}) with 
${\mathbf y}(t) = {\mathbf Q}(t)$ and ${\mathbf B} = {\mathbf M}^{-1}\left ( -{\mathbf M}_\sigma - {\mathbf C} \right )$. Hence all three distinct finite element 
discretizations from Section 2 lead to a semi-discrete system as in (\ref{sdisc_eq3}).

Starting from explicit multi-step Adams-Bashforth methods, we shall now derive explicit local time-stepping schemes of arbitrarily 
high accuracy for a general problem of the form
\begin{equation} \label{sd_mod_pr}
\frac{d {\mathbf y}}{d t}(t) = {\mathbf B} {\mathbf y}(t)\,,  \qquad t \in (0, T)\,.
\end{equation}

\subsection{Adams-Bashforth methods}
First, we briefly recall the construction of the classical $k$-step ($k$th-order) Adams-Bashforth method for the numerical solution of 
(\ref{sd_mod_pr}) \cite{HNW00}. Let $t_i = i \Delta t$ and ${\mathbf y}_n$, ${\mathbf y}_{n-1}$,..., ${\mathbf y}_{n-k+1}$ the numerical approximations 
to the exact solution ${\mathbf y}(t_n)$,..., ${\mathbf y}(t_{n-k+1})$. The solution of (\ref{sd_mod_pr}) satisfies
\begin{equation} \label{int_eq}
{\mathbf y}(t_n + \xi \Delta t) = {\mathbf y}(t_n) + \int_{t_n}^{t_n + \xi \Delta t} {\mathbf B} {\mathbf y}(t) \, dt \,, \qquad 0 < \xi \leq 1\,.
\end{equation}
We now replace the unknown solution ${\mathbf y}(t)$ under the integral in (\ref{int_eq}) by the interpolation polynomial $p(t)$ through the points 
$(t_i, {\mathbf y}_i)$, $i = n-k+1, \dots, n$. It is explicitly given in terms of backward differences 
$$
\nabla^0 {\mathbf y}_n = {\mathbf y}_n\,, \quad \nabla^{j+1} {\mathbf y}_n = \nabla^{j} {\mathbf y}_n - \nabla^{j} {\mathbf y}_{n-1}
$$
by
\begin{equation*}
p(t) = p(t_n + s \Delta t) = \sum_{j = 0}^{k-1}(-1)^j \left (\begin{array}{c} -s \\ j \end{array} \right ) \nabla^{j} {\mathbf y}_n \,.
\end{equation*}
Integration of (\ref{int_eq}) with ${\mathbf y}(t)$ replaced by $p(t)$ then yields the approximation ${\mathbf y}_{n+\xi}$ of 
${\mathbf y}(t_n + \xi \Delta t)$, $0 < \xi \leq 1$,
\begin{equation} \label{ab}
{\mathbf y}_{n+\xi} = {\mathbf y}_n + \Delta t {\mathbf B} \sum_{j=0}^{k-1} \gamma_j (\xi) \nabla^{j} {\mathbf y}_n\,,
\end{equation}
where the polynomials $\gamma_j(\xi)$ are defined as 
$$
\gamma_j(\xi) = (-1)^j \int_0^{\xi} \left ( \begin{array}{c} -s \\ j \end{array} \right ) \, ds \,.
$$
They are given in Table \ref{tab1} for $j \leq 3$. After expressing the backward differences in terms of ${\mathbf y}_{n-j}$ and setting $\xi = 1$ in 
(\ref{ab}), we recover the common form of the $k$-step Adams-Bashforth scheme \cite{HNW00}
\begin{equation} \label{stand_ab}
{\mathbf y}_{n+1} = {\mathbf y}_n + \Delta t {\mathbf B} \sum_{j=0}^{k-1} \alpha_j {\mathbf y}_{n-j} \,,
\end{equation}
where the coefficients $\alpha_j$, $j = 0,\dots, k-1$ for the second, third- and fourth-order ($k=2,3,4$) Adams-Bashforth schemes are given in Table
\ref{tab2}. For higher values of $k$ we refer to \cite{HNW00}.

\begin{table}[h!]
\begin{center}
\renewcommand{\arraystretch}{1.5}
\begin{tabular}{c|cccc} \hline
$j$ & 0 & 1 & 2 & 3  \\ \hline $\gamma_j(\xi)$ & $\xi$ &
$\frac{1}{2} \xi^2$ & $\frac{1}{6} \xi^3 + \frac{1}{4} \xi^2$ &
$\frac{1}{24} \xi^4 + \frac{1}{6} \xi^3 + \frac{1}{6} \xi^2$ \\
\hline
\end{tabular} 
\caption{Coefficients $\gamma_j(\xi)$ for the explicit Adams-Bashforth methods.} \label{tab1}
\end{center}
\end{table}
\begin{table}[ht!]
\begin{center}
\renewcommand{\arraystretch}{1.5}
\begin{tabular}{c|cccc} \hline
      & $\alpha_0$      & $\alpha_1$       & $\alpha_2$      & $\alpha_3$      \\ \hline
$k=2$ & $\frac{3}{2}$   & $-\frac{1}{2}$   & 0               & 0
\\ \hline $k=3$ & $\frac{23}{12}$ & $-\frac{16}{12}$ &
$\frac{5}{12}$  & 0               \\ \hline $k=4$ & $\frac{55}{24}$
& $-\frac{59}{24}$ & $\frac{37}{24}$ & $-\frac{9}{24}$ \\ \hline
\end{tabular} 
\caption{Coefficients for the $k$-th order Adams-Bashforth methods.} \label{tab2}
\end{center}
\end{table}

\subsection{Adams-Bashforth based LTS}
Starting from the classical Adams-Bashforth methods from Section 3.1, we shall now derive LTS schemes of arbitrarily high accuracy for (\ref{sd_mod_pr}), which allow arbitrarily 
small time-steps precisely where small elements in the spatial mesh are located. To do so, we first split the unknown vector ${\mathbf y}(t)$ in two parts
\begin{equation*}
{\mathbf y}(t) = ({\mathbf I} - {\mathbf P}) {\mathbf y}(t) + {\mathbf P} {\mathbf y}(t) = 
{\mathbf y}^{[\mbox{\scriptsize{coarse}}]}(t) + {\mathbf y}^{[\mbox{\scriptsize fine}]}(t)\,,
\end{equation*}
where the matrix $\mathbf P$ is diagonal. Its diagonal entries, equal to zero or one, identify the unknowns associated with the locally refined region, where smaller 
time-steps are needed. Hence $\mathbf P$ corresponds to a discrete partition of unity of the degrees of freedom associated with $V_h$ .

The exact solution of (\ref{sd_mod_pr}) again satisfies
\begin{equation} \label{tmp1}
{\mathbf y}(t_n + \xi \Delta t) = {\mathbf y}(t_n) + \int_{t_n}^{t_n + \xi \Delta t} {\mathbf B} 
\left ( {\mathbf y}^{[\mbox{\scriptsize{coarse}}]}(t) + {\mathbf y}^{[\mbox{\scriptsize fine}]}(t) \right ) \, 
dt \,, \qquad 0 < \xi \leq 1\,.
\end{equation}
Since we wish to use the standard $k$-step Adams-Bashforth method in the coarse region, we approximate the term in (\ref{tmp1}) that involve ${\mathbf
y}^{[\mbox{\scriptsize coarse}]}(t)$ as in (\ref{int_eq}), which yields
\begin{equation} \label{tmp2}
{\mathbf y}(t_n + \xi \Delta t) \approx {\mathbf y}_n + \Delta t \, {\mathbf B} 
({\mathbf I} - {\mathbf P}) \sum_{j=0}^{k-1} \gamma_j (\xi) \nabla^{j} {\mathbf y}_n + \int_{t_n}^{t_n + \xi \Delta t}
{\mathbf B}{\mathbf P} {\mathbf y}(t) \, dt \,.
\end{equation}
To circumvent the severe stability constraint due to the smallest elements associated with ${\mathbf y}^{[\mbox{\scriptsize fine}]}(t)$, we shall now treat 
${\mathbf y}^{[\mbox{\scriptsize fine}]}(t)$ differently from ${\mathbf y}^{[\mbox{\scriptsize coarse}]}(t)$ and instead we approximate the integrand in (\ref{tmp2}) as
\begin{equation*}
\int_{t_n}^{t_n + \xi \Delta t} {\mathbf B}{\mathbf P} {\mathbf y}(t) \, dt \approx 
\int_{0}^{\xi \Delta t} {\mathbf B}{\mathbf P} {\widetilde{\mathbf y}}(\tau) \, d \tau \,,
\end{equation*}
where ${\widetilde {\mathbf y}}(\tau)$ solves the differential
equation
\begin{equation} \label{mod_pr} \begin{split}
\frac{d {\widetilde {\mathbf y}}}{d \tau}(\tau) & = 
{\mathbf B}({\mathbf I} - {\mathbf P}) \sum_{j=0}^{k-1} {\widetilde \gamma}_j
\left ( \frac{\tau}{\Delta t} \right ) \nabla^{j} {\mathbf y}_n +
{\mathbf B}{\mathbf P} \, {\widetilde {\mathbf y}}(\tau) \,,\\
{\widetilde {\mathbf y}}(0) & = {\mathbf y}_n\,,
\end{split} \end{equation}
with coefficients 
\begin{equation} \label{coeff_gamma_lts}
{\widetilde \gamma}_j (\xi) = \frac{d}{d \xi} \gamma_j ({\xi}) = \frac{d}{d \xi} \left ( (-1)^j \int_0^\xi \binom{-s}{j} \, ds \right ) = 
(-1)^j \binom{-\xi}{j}\,.
\end{equation}
The polynomials ${\widetilde \gamma}_j (\xi)$ are given in Table \ref{tab3} for $j \leq 3$. 
\begin{table}[t!]
\begin{center}
\renewcommand{\arraystretch}{1.5}
\begin{tabular}{c|cccc} \hline
$j$ & 0 & 1 & 2 & 3  \\ \hline ${\widetilde \gamma}_j(\xi)$ & $1$ &
$\xi$ & $\frac{1}{2} \xi^2 + \frac{1}{2} \xi$ & $\frac{1}{6} \xi^3 +
\frac{1}{2} \xi^2 + \frac{1}{3} \xi$ \\ \hline
\end{tabular} 
\caption{The polynomial coefficients ${\widetilde \gamma_j}(\xi)$} \label{tab3}
\end{center}
\end{table}
Replacing ${\mathbf y}(t)$ by ${\widetilde {\mathbf y}}(\tau)$ in (\ref{tmp2}), we obtain
\begin{equation} \label{tmp4}
{\mathbf y}(t_n + \xi \Delta t) \approx {\mathbf y}_n + \Delta t \, {\mathbf B} 
({\mathbf I} - {\mathbf P}) \sum_{j=0}^{k-1} \gamma_j (\xi) \nabla^{j} {\mathbf y}_n + 
\int_{0}^{\xi \Delta t} {\mathbf B}{\mathbf P} {\widetilde{\mathbf y}}(\tau) \, d \tau \,.
\end{equation}
By considering (\ref{mod_pr}) in integrated form, we find that
\begin{equation} \label{tmp3} \begin{split}
{\widetilde {\mathbf y}}(\xi \Delta t) & = 
{\widetilde {\mathbf y}}(0) + {\mathbf B}({\mathbf I} - {\mathbf P}) 
\sum_{j=0}^{k-1} \left ( 
\int_0^{\xi \Delta t} {\widetilde \gamma}_j \left( \frac{\tau}{\Delta t} \right) \, d \tau \right ) 
\nabla^{j} {\mathbf y}_n 
+ \int_{0}^{\xi \Delta t} {\mathbf B}{\mathbf P} \, {\widetilde {\mathbf y}}(\tau) \, d \tau \\
& = {\mathbf y}_n + \Delta t \, {\mathbf B}({\mathbf I} - {\mathbf P}) 
\sum_{j=0}^{k-1}  \gamma_j (\xi) \nabla^{j} {\mathbf y}_n + 
\int_{0}^{\xi \Delta t} {\mathbf B}{\mathbf P} \, {\widetilde {\mathbf y}}(\tau) \, d \tau \,.
\end{split} \end{equation}
 From the comparison of (\ref{tmp4}) and (\ref{tmp3}) we infer that
$${\mathbf y}(t_n + \xi \Delta t) \approx {\widetilde {\mathbf y}}(\xi \Delta t) \,.$$
Thus to advance ${\mathbf y}(t_n)$ from $t_n$ to $t_n + \Delta t$, we shall evaluate $\widetilde{\mathbf y}(\Delta t)$ by 
solving (\ref{mod_pr}) on $[0, \Delta t]$ numerically. We solve (\ref{mod_pr}) until $\tau = \Delta t$ again with a $k$-step
Adams-Bashforth scheme, using a smaller time-step $\Delta \tau = \Delta t / p$, where $p$ denotes the ratio of local 
refinement. For $m = 0, \dots, p-1$ we then have
\begin{equation} \label{tmp5}
\widetilde{\mathbf y}_{(m+1)/p} = \widetilde{\mathbf y}_{m/p} +
\Delta \tau \, {\mathbf B}({\mathbf I} - {\mathbf P}) 
\sum_{\ell = 0}^{k-1} \alpha_{\ell} \sum_{j = 0}^{k-1} {\widetilde \gamma}_j \left ( \frac{m-\ell}{p} \right ) 
\nabla^{j} {\mathbf y}_n +
\Delta \tau \, {\mathbf B}{\mathbf P} \sum_{\ell = 0}^{k-1} \alpha_{\ell} {\widetilde{\mathbf y}}_{(m-l)/p} \,,
\end{equation}
where $\alpha_\ell$, $\ell = 0,\dots, k-1$ denote the coefficients of the classical $k$-step Adams-Bashforth scheme 
(see Table \ref{tab2}). Finally, after expressing the backward differences in terms of 
${\mathbf y}_{n-\ell}$, we find
\begin{equation} \label{lts_ab}
\widetilde{\mathbf y}_{(m+1)/p} = \widetilde{\mathbf y}_{m/p} +
\Delta \tau \, {\mathbf B}({\mathbf I} - {\mathbf P}) \sum_{\ell = 0}^{k-1} \beta_{m, \ell} \, {\mathbf y}_{n-\ell} + 
\Delta \tau \, {\mathbf B}{\mathbf P} \sum_{\ell = 0}^{k-1} \alpha_{\ell} \, {\widetilde{\mathbf y}}_{(m-l)/p} \,, \ \ m = 0, \dots, p-1,
\end{equation}
where the constant coefficients $\beta_{m, \ell}$, $m=0,\dots,p-1$, $\ell=
0,\dots,k-1$, satisfy
\begin{equation} \label{coeff_beta_lts}
\beta_{m, \ell} = \sum_{i=0}^{k-1} \alpha_i 
\sum_{j = \ell}^{k-1} (-1)^{\ell} \left ( \begin{array}{c} j \\ \ell \end{array} \right )
{\widetilde \gamma}_j \left ( \frac{m-i}{p} \right ) \,,
\end{equation}
with ${\widetilde \gamma}_j$ defined in (\ref{coeff_gamma_lts}).

In summary, the LTS-AB$k$($p$) algorithm computes ${\mathbf y}_{n+1} \simeq \mathbf y(t_n +
\Delta t)$, given $\mathbf y_n$, ${\mathbf y}_{n-1}$,..., ${\mathbf y}_{n-k+1}$, 
${\mathbf B}({\mathbf I} - {\mathbf P}) {\mathbf y}_{n-1}$,..., ${\mathbf B}({\mathbf I} - {\mathbf P}) {\mathbf y}_{n-k+1}$ 
and ${\mathbf P}{\mathbf y}_{n-1/p}$, ${\mathbf P}{\mathbf y}_{n-2/p}$,..., ${\mathbf P}{\mathbf y}_{n-(k-1)/p}$ as follows:

\bigskip

\noindent {\bf LTS-AB$k$($p$) Algorithm}{\it 
\begin{enumerate}
\item Set ${\widetilde{\mathbf y}}_0:= \mathbf y_n$, ${\widetilde{\mathbf y}}_{-\ell/p}:= {\mathbf P}{\mathbf y}_{n-\ell/p}$, 
$\ell = 1, \dots, k-1$.
\item Set ${\mathbf w}_{n-\ell} := {\mathbf B}({\mathbf I} - {\mathbf P}) {\mathbf y}_{n-\ell}$, $\ell = 1, \dots, k-1$.
\item Compute ${\mathbf w}_{n} := {\mathbf B}({\mathbf I} - {\mathbf P}) {\mathbf y}_{n}$.
\item For $m = 0, \dots, p-1$, compute
\begin{equation*}
\widetilde{\mathbf y}_{(m+1)/p} := \widetilde{\mathbf y}_{m/p} + \frac{\Delta t}{p} 
\sum_{\ell = 0}^{k-1} \beta_{m, \ell} \, {\mathbf w}_{n-\ell} + 
\frac{\Delta t}{p} {\mathbf B}{\mathbf P} \sum_{\ell = 0}^{k-1} \alpha_{\ell} \, {\widetilde{\mathbf y}}_{(m-l)/p} \,.
\end{equation*}
\item Set ${\mathbf y}_{n+1} := \widetilde{\mathbf y}_1$.
\end{enumerate}}

Steps 1-4 correspond to the numerical solution of (\ref{mod_pr}) until $\tau = \Delta t$ with the $k$-step 
Adams-Bashforth scheme, using the local time-step $\Delta \tau = \Delta t / p$. For $\mathbf P = \mathbf 0$ or $p=1$, that is without 
any local time-stepping, we thus recover the standard $k$-step Adams-Bashforth scheme. If the fraction of nonzero entries in
$\mathbf P$ is small, the overall cost is dominated by the computation of ${\mathbf w}_{n}$ in Step 3,  which requires one 
multiplications by ${\mathbf B}({\mathbf I} - {\mathbf P})$ per time-step $\Delta t$. All further matrix-vector 
multiplications by ${\mathbf B} {\mathbf P}$ only affect those unknowns that lie inside the refined region, or immediately 
next to it; hence, their computational cost remains negligible as long as the locally refined region contains a small part of $\Omega$.

\subsection{Examples of LTS-AB$k$($p$) schemes}
We have shown above how to derive LTS-AB$k$($p$) schemes of arbitrarily high accuracy. Since the third- and fourth-order LTS-AB$k$($p$) 
schemes are probably the most relevant for applications, we now describe in detail the LTS-AB$k$($p$) schemes for $k=3$, 4 and $p=2$.

For $k=3$ and $p=2$, we find from (\ref{tmp5}) and (\ref{lts_ab}) for the first half-step of size $\Delta \tau = \Delta t / 2$ that 
\begin{equation*} \begin{split}
\widetilde{\mathbf y}_{1/2} & = {\widetilde{\mathbf y}}_0 +
\frac{\Delta t}{2} {\mathbf B}({\mathbf I} - {\mathbf P}) \Bigl [
\alpha_0 \left ( \widetilde{\gamma}_0(0) \nabla^0 {\mathbf y}_n + \widetilde{\gamma}_1(0) \nabla^1 {\mathbf y}_n 
+ \widetilde{\gamma}_2(0) \nabla^2 {\mathbf y}_n \right )  \\
& \qquad \qquad \qquad \qquad \ + \alpha_1 \left ( \widetilde{\gamma}_0(-1/2) \nabla^0 {\mathbf y}_n + 
\widetilde{\gamma}_1(-1/2) \nabla^1 {\mathbf y}_n 
+ \widetilde{\gamma}_2(-1/2) \nabla^2 {\mathbf y}_n \right ) \\
& \qquad \qquad \qquad \qquad \ + \alpha_2 \left ( \widetilde{\gamma}_0(-1) \nabla^0 {\mathbf y}_n + 
\widetilde{\gamma}_1(-1) \nabla^1 {\mathbf y}_n 
+ \widetilde{\gamma}_2(-1) \nabla^2 {\mathbf y}_n \right ) \Bigr ] \\
& \qquad \ \ + \frac{\Delta t}{2}{\mathbf B}{\mathbf P} \Bigl [ \alpha_0 {\widetilde{\mathbf y}}_0 +
\alpha_1 {\widetilde{\mathbf y}}_{-1/2} + \alpha_2 {\widetilde{\mathbf y}}_{-1} \Bigr ] \\
& = {{\mathbf y}}_n + \frac{\Delta t}{2} {\mathbf B}({\mathbf I} - {\mathbf P}) \Bigl [ 
{\beta_{0, 0}} {\mathbf y}_n + {\beta_{0, 1}} {\mathbf y}_{n-1} + {\beta_{0, 2}} {\mathbf y}_{n-2} \Bigl ] 
+ \frac{\Delta t}{2}{\mathbf B}{\mathbf P} \Bigl [ \alpha_0 {{\mathbf y}}_n + 
\alpha_1 {\widetilde{\mathbf y}}_{-1/2} + \alpha_2 {{\mathbf y}}_{n-1} \Bigr ]
\end{split} \end{equation*}
with
\begin{equation*} \begin{split}
\beta_{0, 0} & = \alpha_0 \left \{ \widetilde{\gamma}_0(0) + \widetilde{\gamma}_1(0) + \widetilde{\gamma}_2(0) \right \} +
                 \alpha_1 \left \{ \widetilde{\gamma}_0(-1/2) + \widetilde{\gamma}_1(-1/2) + \widetilde{\gamma}_2(-1/2) \right \} \\
                 & + \alpha_2 \left \{ \widetilde{\gamma}_0(-1) + \widetilde{\gamma}_1(-1) + \widetilde{\gamma}_2(-1) \right \}\\
             & = \frac{23}{12} \left \{ 1 + 0 + 0 \right \} - \frac{16}{12} \left \{ 1 - \frac{1}{2} - \frac{1}{8}\right \} 
               + \frac{5}{12} \left \{ 1 - 1 + 0 \right \} = \frac{17}{12} \,, \\
\beta_{0, 1} & = \alpha_0 \left \{ - \widetilde{\gamma}_1(0) -2 \widetilde{\gamma}_2(0) \right \} +
                 \alpha_1 \left \{ - \widetilde{\gamma}_1(-1/2) -2 \widetilde{\gamma}_2(-1/2) \right \} 
             + \alpha_2 \left \{ - \widetilde{\gamma}_1(-1) -2 \widetilde{\gamma}_2(-1) \right \} \\
             & = \frac{23}{12} \left \{ 0 + 0 \right \} - \frac{16}{12} \left \{ \frac{1}{2} + \frac{1}{4}\right \} 
               + \frac{5}{12} \left \{ 1 + 0 \right \} = -\frac{7}{12} \,, \\
\beta_{0, 2} & = \alpha_0  \widetilde{\gamma}_2(0) + \alpha_1 \widetilde{\gamma}_2(-1/2) + \alpha_2  \widetilde{\gamma}_2(-1)
               = \frac{23}{12} \cdot 0 - \frac{16}{12} \left \{ -\frac{1}{8} \right \} + \frac{5}{12} \cdot 0 = \frac{2}{12} \,.
\end{split} \end{equation*}
\begin{figure}[t!]
\centerline{\includegraphics[scale=0.7]{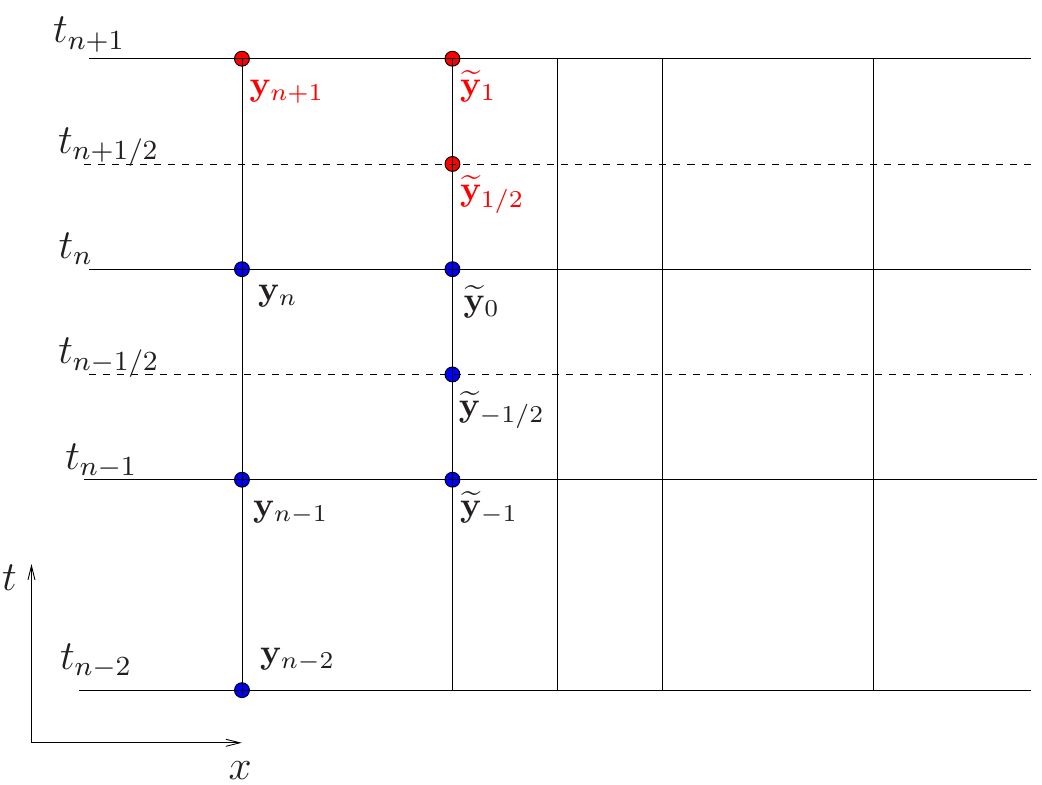}}  
\caption{Illustration of the LTS-AB$3$($2$) scheme in one space dimension.} \label{fig_lts32}
\end{figure}
Next, we perform a second half-step thus completing the full step of size $\Delta t$:
\begin{equation*} \begin{split}
\widetilde{\mathbf y}_{1} & = {\widetilde{\mathbf y}}_{1/2} +
\frac{\Delta t}{2} {\mathbf B}({\mathbf I} - {\mathbf P}) \Bigl [
\alpha_0 \left ( \widetilde{\gamma}_0(1/2) \nabla^0 {\mathbf y}_n + \widetilde{\gamma}_1(1/2) \nabla^1 {\mathbf y}_n 
+ \widetilde{\gamma}_2(1/2) \nabla^2 {\mathbf y}_n \right )  \\
& \qquad \qquad \qquad \qquad \ + \alpha_1 \left ( \widetilde{\gamma}_0(0) \nabla^0 {\mathbf y}_n + \widetilde{\gamma}_1(0) \nabla^1 {\mathbf y}_n 
+ \widetilde{\gamma}_2(0) \nabla^2 {\mathbf y}_n \right ) \\
& \qquad \qquad \qquad \qquad \ + \alpha_2 \left ( \widetilde{\gamma}_0(-1/2) \nabla^0 {\mathbf y}_n + \widetilde{\gamma}_1(-1/2) \nabla^1 {\mathbf y}_n 
+ \widetilde{\gamma}_2(-1/2) \nabla^2 {\mathbf y}_n \right ) \Bigr ] \\
& \qquad \ \ + \frac{\Delta t}{2}{\mathbf B}{\mathbf P} \Bigl [ \alpha_0 {\widetilde{\mathbf y}}_{1/2} + \alpha_1 {\widetilde{\mathbf y}}_{0} + \alpha_2 {\widetilde{\mathbf y}}_{-1/2} \Bigr ] \\
& = {\widetilde {\mathbf y}}_{1/2} + \frac{\Delta t}{2} {\mathbf B}({\mathbf I} - {\mathbf P}) \Bigl [ 
{\beta_{1, 0}} {\mathbf y}_n + {\beta_{1, 1}} {\mathbf y}_{n-1} + {\beta_{1, 2}} {\mathbf y}_{n-2} \Bigl ] 
+ \frac{\Delta t}{2}{\mathbf B}{\mathbf P} \Bigl [ \alpha_0 {\widetilde {\mathbf y}}_{1/2} + 
\alpha_1 {{\mathbf y}}_{n} + \alpha_2 {\widetilde {\mathbf y}}_{-1/2} \Bigr ]
\end{split} \end{equation*}
with
\begin{equation*} \begin{split}
\beta_{1, 0} & = \alpha_0 \left \{ \widetilde{\gamma}_0(1/2) + \widetilde{\gamma}_1(1/2) + \widetilde{\gamma}_2(1/2) \right \} +
                 \alpha_1 \left \{ \widetilde{\gamma}_0(0) + \widetilde{\gamma}_1(0) + \widetilde{\gamma}_2(0) \right \} \\
              &+ \alpha_2 \left \{ \widetilde{\gamma}_0(-1/2) + \widetilde{\gamma}_1(-1/2) + \widetilde{\gamma}_2(-1/2) \right \} \\
             & = \frac{23}{12} \left \{ 1 + \frac{1}{2} + \frac{3}{8} \right \} - \frac{16}{12} \left \{ 1 + 0 + 0 \right \} 
               + \frac{5}{12} \left \{ 1 - \frac{1}{2} - \frac{1}{8} \right \} = \frac{29}{12} \,,\\
\beta_{1, 1} & = \alpha_0 \left \{ - \widetilde{\gamma}_1(1/2) -2 \widetilde{\gamma}_2(1/2) \right \} +
                 \alpha_1 \left \{ - \widetilde{\gamma}_1(0) -2 \widetilde{\gamma}_2(0) \right \}
               + \alpha_2 \left \{ - \widetilde{\gamma}_1(-1/2) -2 \widetilde{\gamma}_2(-1/2) \right \} \\
             & = \frac{23}{12} \left \{ -\frac{1}{2} -\frac{3}{4} \right \} - \frac{16}{12} \left \{ 0 - 0 \right \} 
               + \frac{5}{12} \left \{ \frac{1}{2} + \frac{1}{4} \right \} = -\frac{25}{12} \,, \\
\beta_{1, 2} & = \alpha_0  \widetilde{\gamma}_2(1/2) + \alpha_1 \widetilde{\gamma}_2(0) + \alpha_2  \widetilde{\gamma}_2(-1/2)
               = \frac{23}{12} \left \{ \frac{3}{8} \right \} - \frac{16}{12} \cdot 0 + \frac{5}{12} \left \{ -\frac{1}{8} \right \} = \frac{8}{12} \,.
\end{split} \end{equation*}
Recall that the coefficients $\alpha_\ell$ correspond to the standard coefficients of the Adams-Bashforth methods given in Table \ref{tab2}.
After simplification, the LTS-AB$3$($2$) method then reads (see Figure \ref{fig_lts32}):
\begin{equation*} \begin{split}
\widetilde{\mathbf y}_{1/2} & = {{\mathbf y}}_n + \frac{\Delta t}{2} {\mathbf B}({\mathbf I} - {\mathbf P}) \left [
\frac{17}{12} {\mathbf y}_n - \frac{7}{12} {\mathbf y}_{n-1} +  \frac{2}{12} {\mathbf y}_{n-2} \right ] 
+ \frac{\Delta t}{2}{\mathbf B}{\mathbf P} \left [ \frac{23}{12} {{\mathbf y}}_n -
\frac{16}{12} {\widetilde{\mathbf y}}_{-1/2} + \frac{5}{12} {{\mathbf y}}_{n-1} \right ] \,, \\
{\mathbf y}_{n+1} = \widetilde{\mathbf y}_{1} & = \widetilde{\mathbf y}_{1/2} + \frac{\Delta t}{2} {\mathbf B}({\mathbf I} - {\mathbf
P}) \left [ \frac{29}{12} {\mathbf y}_n - \frac{25}{12} {\mathbf y}_{n-1} +  \frac{8}{12} {\mathbf y}_{n-2} \right ] 
+ \frac{\Delta t}{2}{\mathbf B}{\mathbf P} \left [ \frac{23}{12} 
{\widetilde{\mathbf y}}_{1/2} - \frac{16}{12} {{\mathbf y}}_{n} + \frac{5}{12} {\widetilde{\mathbf y}}_{-1/2} \right ] \,.
\end{split} \end{equation*}

For the case with $k=4$ and $p = 2$, similar calculations yield the LTS-AB$4$($2$) scheme:
\begin{equation*} \begin{split}
\widetilde{\mathbf y}_{1/2}  = {{\mathbf y}}_n & +
\frac{\Delta t}{2} {\mathbf B}({\mathbf I} - {\mathbf P}) \left [
\frac{297}{192} {\mathbf y}_n - \frac{187}{192} {\mathbf y}_{n-1} +  \frac{107}{192} {\mathbf y}_{n-2} - \frac{25}{192} {\mathbf y}_{n-3}\right ] \\
& + \frac{\Delta t}{2}{\mathbf B}{\mathbf P} \left [ \frac{55}{24} {{\mathbf y}}_n 
- \frac{59}{24} {\widetilde{\mathbf y}}_{-1/2} + \frac{37}{24} {{\mathbf y}}_{n-1} - \frac{9}{24} {\widetilde{\mathbf y}}_{-3/2} \right ] \,, \\
{\mathbf y}_{n+1} = \widetilde{\mathbf y}_{1} = \widetilde{\mathbf y}_{1/2} & + \frac{\Delta t}{2} {\mathbf B}({\mathbf I} - {\mathbf P}) \left [
\frac{583}{192} {\mathbf y}_n - \frac{757}{192} {\mathbf y}_{n-1} + \frac{485}{192} {\mathbf y}_{n-2} - \frac{119}{192} {\mathbf y}_{n-3} \right ] \\
& + \frac{\Delta t}{2}{\mathbf B}{\mathbf P} \left [ \frac{55}{24} {\widetilde{\mathbf y}}_{1/2} 
- \frac{59}{24} {{\mathbf y}}_{n} + \frac{37}{24} {\widetilde{\mathbf y}}_{-1/2} - \frac{9}{24} {{\mathbf y}}_{n-1} \right ] \,.
\end{split} \end{equation*}

Other examples of LTS Adams-Bashforth schemes are listed in Table \ref{tab4}, where the coefficients
of the schemes are given for different values of $k$ and $p$.

\begin{table}[ht!]
\begin{center}
\renewcommand{\arraystretch}{1.5}
\begin{tabular}{c|c|cccc} \hline
Scheme     & $m$ & $\beta_{m, 0}$     & $\beta_{m, 1}$      &
$\beta_{m, 2}$     & $\beta_{m, 3}$     \\ \hline
LTS-AB2(2) &  0  & $\frac{5}{4}$      & $-\frac{1}{4}$      &     0              & 0                  \\
           &  1  & $\frac{7}{4}$      & $-\frac{3}{4}$      &     0              & 0                  \\ \hline
LTS-AB2(3) &  0  & $\frac{7}{6}$      & $-\frac{1}{6}$      &     0              & 0                  \\
           &  1  & $\frac{9}{6}$      & $-\frac{3}{6}$      &     0              & 0                  \\
           &  2  & $\frac{11}{6}$     & $-\frac{5}{6}$      &     0              & 0                  \\ \hline \hline
LTS-AB3(2) &  0  & $\frac{17}{12}$    & $-\frac{7}{12}$     & $\frac{2}{12}$     & 0                  \\
           &  1  & $\frac{29}{12}$    & $-\frac{25}{12}$    & $\frac{8}{12}$     & 0                  \\ \hline
LTS-AB3(3) &  0  & $\frac{137}{108}$  & $-\frac{40}{108}$   & $\frac{11}{108}$   & 0                  \\
           &  1  & $\frac{203}{108}$  & $-\frac{136}{108}$  & $\frac{41}{108}$   & 0                  \\
           &  2  & $\frac{281}{108}$  & $-\frac{256}{108}$  & $\frac{83}{108}$   & 0                  \\ \hline \hline
LTS-AB4(2) &  0  & $\frac{297}{192}$  & $-\frac{187}{192}$  & $\frac{107}{192}$  & $-\frac{25}{192}$  \\
           &  1  & $\frac{583}{192}$  & $-\frac{757}{192}$  & $\frac{485}{192}$  & $-\frac{119}{192}$ \\ \hline
LTS-AB4(3) &  0  & $\frac{871}{648}$  & $-\frac{387}{648}$  & $\frac{213}{648}$  & $-\frac{49}{648}$  \\
           &  1  & $\frac{1425}{648}$ & $-\frac{1437}{648}$ & $\frac{867}{648}$  & $-\frac{207}{648}$  \\
           &  2  & $\frac{2159}{648}$ & $-\frac{2955}{648}$ & $\frac{1917}{648}$ & $-\frac{473}{648}$  \\ \hline
\end{tabular} 
\caption{Coefficients of LTS-AB$k$($p$) schemes for $k~=~2, 3, 4$ and $p = 2, 3, 4$.} \label{tab4}
\end{center}
\end{table}

\subsection{Accuracy of the LTS scheme}
We now establish the accuracy of the LTS-AB$k$($p$) scheme. To do so, we first recall that for $k \geq 1$, we have 
\begin{equation} \label{prop_alpha}
\sum_{j=0}^{k-1} \alpha_j = 1\,,
\end{equation}
where $\alpha_j$, $j = 0, \dots, k-1$ are the standard coefficients of the $k$-step Adams-Bashforth scheme (\ref{stand_ab}); 
see Theorem 2.4 in \cite{HNW00} for details. Next, we shall need the following identity.
\begin{lemma}
For $k \geq 1$ and $p \geq 2$ we have
\begin{equation} \label{prop_beta}
\sum_{m=0}^{p-1} \beta_{m, \ell} = p \, \alpha_{\ell}\,, \quad \ell = 0, \dots, k-1\,,
\end{equation}
where $\alpha_\ell$ and $\beta_{m, \ell}$ ($\ell = 0, \dots, k-1$, $m = 0, \dots, p-1$) correspond to the standard 
coefficients of the $k$-step Adams-Bashforth scheme (\ref{stand_ab}) and the coefficients of the LTS-AB$k$($p$) scheme defined in (\ref{coeff_beta_lts}), respectively.
\end{lemma}

\begin{proof} The identity in (\ref{prop_beta}) was verified by computer algebra for all $k \leq 20$ and $p \leq 1000$, which is sufficient for all 
practical purposes. It probably holds for all values of $k$ and $p$. \end{proof}

We are now ready to establish the accuracy of the LTS-AB$k$($p$) scheme (Algorithm 3.1).

\begin{theorem}
The local time-stepping method LTS-AB$k$($p$) is consistent of order $k$.
\end{theorem}

\begin{proof}
For simplicity, we restrict ourselves to the cases $k = 2$, 3 and 4, as the extension to the general case $k > 4$ 
is straightforward but cumbersome.

To prove the second-order consistency of LTS-AB$2$($p$), we need to show that the local error ${\mathbf y}_{n+1} - {\mathbf y}(t_{n+1})$ is 
${\mathcal O}(\Delta t^3)$. Since (\ref{lts_ab}) with $k=2$ and 
$\Delta \tau = \Delta t / p$ holds for all $m = 0, \dots, p-1$, we have that
\begin{equation*} \begin{split}
{\mathbf y}_{n+1} = {\widetilde {\mathbf y}}_1 = {\widetilde {\mathbf y}}_0 & + 
\frac{\Delta t}{p} {\mathbf B}({\mathbf I} - {\mathbf P}) 
\left ( \sum_{m=0}^{p-1} \beta_{m, 0} {\mathbf y}_n + \sum_{m=0}^{p-1} \beta_{m, 1} {\mathbf y}_{n-1} \right ) \\
& + \frac{\Delta t}{p} {\mathbf B}{\mathbf P} 
\left ( \alpha_0 {\widetilde {\mathbf y}}_{(p-1)/p} + \sum_{m=0}^{p-2} (\alpha_0 + \alpha_1) {\widetilde {\mathbf y}}_{m/p}
+ \alpha_1 {\widetilde {\mathbf y}}_{-1/p} \right ) \,.
\end{split} \end{equation*}
Next, we use that ${\widetilde{\mathbf y}}_0:= \mathbf y_n$ as well as (\ref{prop_alpha}) and (\ref{prop_beta}) with $k=2$. This yields
\begin{equation} \label{lts-ab2-1} 
{\mathbf y}_{n+1} = {\mathbf y}_n + \Delta t \, {\mathbf B}({\mathbf I} - {\mathbf P}) 
\left ( \alpha_0 {\mathbf y}_n + \alpha_1 {\mathbf y}_{n-1} \right ) + \frac{\Delta t}{p} {\mathbf B}{\mathbf P} 
\left ( \alpha_0 {\widetilde {\mathbf y}}_{(p-1)/p} + \sum_{m=0}^{p-2}{\widetilde {\mathbf y}}_{m/p} + \alpha_1 {\widetilde {\mathbf y}}_{-1/p} \right )\,.
\end{equation}
For $\tau = 0$ and $k=2$, we find from (\ref{mod_pr}) and ${\widetilde \gamma}_j(0) = 0$, $j \geq 1$, that
$$
{\widetilde {\mathbf y}}'(0) = {\mathbf B}({\mathbf I} - {\mathbf P}) {\mathbf y}_n + 
{\mathbf B}{\mathbf P} {\widetilde {\mathbf y}}(0) = {\mathbf B}{\mathbf y}_n = {\mathbf y}'(t_n)\,.
$$
Thus, we may expand ${\widetilde {\mathbf y}}_{m/p}$, $m = -1, \dots, p-1$, in Taylor series as
\begin{equation} \label{lts-ab2-2}
{\widetilde {\mathbf y}}_{m/p} = {\widetilde {\mathbf y}}(0) + \frac{m}{p} \, \Delta t \, {\widetilde {\mathbf y}}'(0) + {\mathcal O}(\Delta t^2) = 
{\mathbf y}_n +  \frac{m}{p} \, \Delta t \, {\mathbf y}'(t_n) + {\mathcal O}(\Delta t^2)\,.
\end{equation}
We now insert (\ref{lts-ab2-2}) into (\ref{lts-ab2-1}), replace ${\mathbf y}_{n-1}$ by its Taylor expansion, use 
(\ref{prop_alpha}) with $k=2$ and obtain after some simplifications
\begin{equation*}
{\mathbf y}_{n+1} = {\mathbf y}_n+ \Delta t \, {\mathbf B}({\mathbf I} - {\mathbf P})
\left ( {\mathbf y}_n - \alpha_1 \, \Delta t \, {\mathbf y}'(t_n) + {\mathcal O}(\Delta t^2) \right )
+ \Delta t \, {\mathbf B}{\mathbf P} \left ( {\mathbf y}_n + {\mathcal C}_1 \, \Delta t \, {\mathbf y}'(t_n) 
+ {\mathcal O}(\Delta t^2) \right ) \,,
\end{equation*}
where
$$
{\mathcal C}_1 := \frac{1}{p^2} \left [ \alpha_0 (p-1) + \frac{(p-1)(p-2)}{2} - \alpha_1 \right ]\,.
$$
With the coefficients of the classical two-step Adams-Bashforth scheme $\alpha_0 = 3/2$ and $\alpha_1 = -1/2$, we note that
$$
\alpha_1 + {\mathcal C}_1 = \frac{1}{2p^2} \left [ -p^2 + 3(p-1) + (p-1)(p-2) + 1 \right ] = 0\,.
$$
By differentiating (\ref{sd_mod_pr}), we thus find
\begin{equation*}
{\mathbf y}_{n+1} = {\mathbf y}_n + \Delta t \, {\mathbf B} 
\left ( {\mathbf y}_n - \alpha_1 \, \Delta t \, {\mathbf y}'(t_n) + {\mathcal O}(\Delta t^2) \right ) = 
{\mathbf y}_n + \Delta t \, {\mathbf y}'(t_n) + \frac{\Delta t^2}{2}\,{\mathbf y}''(t_n) + {\mathcal O}(\Delta t^3)\,,
\end{equation*}
which yields a local error of order ${\mathcal O}(\Delta t^3)$.

For $k = 3$, we need to prove a local error of order ${\mathcal O}(\Delta t^4)$. Similarly to the derivation of (\ref{lts-ab2-1}), we now find
\begin{equation} \label{lts-ab3-1} \begin{split}
{\mathbf y}_{n+1} & = {\mathbf y}_n + \Delta t \, {\mathbf B}({\mathbf I} - {\mathbf P}) 
\left ( \alpha_0 {\mathbf y}_n + \alpha_1 {\mathbf y}_{n-1} + \alpha_2 {\mathbf y}_{n-2} \right ) \\
& + \frac{\Delta t}{p} {\mathbf B}{\mathbf P} 
\left ( \alpha_0 {\widetilde {\mathbf y}}_{(p-1)/p} + (\alpha_0 + \alpha_1){\widetilde {\mathbf y}}_{(p-2)/p} 
+ \sum_{m=0}^{p-3}{\widetilde {\mathbf y}}_{m/p}
+ (\alpha_1+\alpha_2) {\widetilde {\mathbf y}}_{-1/p} + \alpha_2 {\widetilde {\mathbf y}}_{-2/p} \right )\,.
\end{split} \end{equation}
Again, from (\ref{mod_pr}) we have ${\widetilde {\mathbf y}}'(0) = {\mathbf y}'(t_n)$. By differentiation of (\ref{mod_pr}) and 
using Taylor expansions we also find
\begin{equation*} \begin{split}
{\widetilde {\mathbf y}}''(0) & = {\mathbf B}({\mathbf I} - {\mathbf P}) 
\sum_{j=0}^{2} \frac{1}{\Delta t} \, {\widetilde \gamma}_j'(0) \nabla^j {\mathbf y}_n + {\mathbf B}{\mathbf P} {\widetilde {\mathbf y}}'(0) 
= {\mathbf B}({\mathbf I} - {\mathbf P}) \frac{1}{\Delta t} \left ( \frac{3}{2} {\mathbf y}_n -2 {\mathbf y}_{n-1} + \frac{1}{2}{\mathbf y}_{n-2} \right ) 
+ {\mathbf B}{\mathbf P} {\mathbf y}'(t_n) \\
& = {\mathbf B}({\mathbf I} - {\mathbf P}) \frac{1}{\Delta t} 
\left ( \Delta t \, {\mathbf y}'(t_n) + {\mathcal O}(\Delta t^2) \right ) + {\mathbf B}{\mathbf P} {\mathbf y}'(t_n) 
= {\mathbf y}''(t_n) + {\mathcal O}(\Delta t)\,.
\end{split} \end{equation*}
Thus, we deduce that
$$
{\widetilde {\mathbf y}}_{m/p} = {\mathbf y}_n + \frac{m}{p} \, \Delta t \, {\mathbf y}'(t_n) + 
\frac{m^2}{p^2} \, \frac{\Delta t^2}{2} \, {\mathbf y}''(t_n) + {\mathcal O}(\Delta t^3)\,, \quad m = -2, \dots, p-1\,.
$$
By following similar arguments as for $k=2$, we now obtain
\begin{equation} \label{tmp_eq_k3} \begin{split}
{\mathbf y}_{n+1} = {\mathbf y}_n & + \Delta t \, {\mathbf B}({\mathbf I} - {\mathbf P})
\left ( {\mathbf y}_n -(\alpha_1+2 \alpha_2) \, \Delta t \, {\mathbf y}'(t_n)
+ (\alpha_1+4 \alpha_2) \, \frac{\Delta t^2}{2} \, {\mathbf y}''(t_n) + {\mathcal O}(\Delta t^3) \right ) \\
& + \Delta t \, {\mathbf B}{\mathbf P} \left ( {\mathbf y}_n + {\mathcal C}_1 \, \Delta t \, {\mathbf y}'(t_n) 
+ {\mathcal C}_2 \, \frac{\Delta t^2}{2} \, {\mathbf y}''(t_n) + {\mathcal O}(\Delta t^3) \right ) \,,
\end{split} \end{equation}
where
\begin{equation*} \begin{split}
{\mathcal C}_1 & :=  \frac{1}{p^2} \left [ \alpha_0 (p-1) + (\alpha_0 + \alpha_1)(p-2) + \frac{(p-2)(p-3)}{2} 
- (\alpha_1 + \alpha_2) - 2 \alpha_2 \right ]\,, \\
{\mathcal C}_2 & :=  \frac{1}{p^3} \left [ \alpha_0 (p-1)^2 + (\alpha_0 + \alpha_1)(p-2)^2 + \frac{(p-2)(p-3)(2p-5)}{6} 
+ (\alpha_1 + \alpha_2) + 4 \alpha_2 \right ]\,.
\end{split} \end{equation*}
With the coefficients of the classical three-step Adams-Bashforth scheme $\alpha_0 = 23/12$, $\alpha_1 = -16/12$ and 
$\alpha_2 = 5/12$, we can easily verify that
$$
\alpha_1 + 2 \alpha_2 + {\mathcal C}_1 = 0\,, \quad -\alpha_1 - 4 \alpha_2 + {\mathcal C}_2 = 0\,, \quad 
-\alpha_1 - 2 \alpha_2 = \frac{1}{2}\,, \quad \alpha_1 + 4 \alpha_2 = \frac{1}{3}\,.
$$
Hence (\ref{tmp_eq_k3}) reduces to
$$
{\mathbf y}_{n+1} = {\mathbf y}_n + \Delta t \, {\mathbf y}'(t_n) + \frac{\Delta t^2}{2}\,{\mathbf y}''(t_n) 
+ \frac{\Delta t^3}{6}\,{\mathbf y}'''(t_n) + {\mathcal O}(\Delta t^4)\,,
$$
which completes the proof for the LTS-AB$3$($p$) scheme.

The case $k=4$ follows from similar arguments. For $k=4$, we find from (\ref{mod_pr}) with $\tau = 0$ that 
${\widetilde {\mathbf y}}'(0) = {\mathbf y}'(t_n)$ and by using Taylor expansions that
\begin{equation*} \begin{split}
{\widetilde {\mathbf y}}''(0) & = {\mathbf B}({\mathbf I} - {\mathbf P}) \frac{1}{\Delta t} 
\left ( \frac{11}{6} {\mathbf y}_n -3 {\mathbf y}_{n-1} + \frac{3}{2}{\mathbf y}_{n-2} - \frac{1}{3}{\mathbf y}_{n-3} \right ) 
+ {\mathbf B}{\mathbf P} {\mathbf y}'(0) = {\mathbf y}''(t_n) + {\mathcal O}(\Delta t)\,,\\
{\widetilde {\mathbf y}}'''(0) & = {\mathbf B}({\mathbf I} - {\mathbf P}) \frac{1}{\Delta t^2} 
\left ( 2{\mathbf y}_n -5 {\mathbf y}_{n-1} + 4 {\mathbf y}_{n-2} - {\mathbf y}_{n-3} \right ) 
+ {\mathbf B}{\mathbf P} {\mathbf y}''(0) = {\mathbf y}'''(t_n) + {\mathcal O}(\Delta t)\,.
\end{split} \end{equation*}
This yields
\begin{equation} \label{lts-ab4-1}
{\widetilde {\mathbf y}}_{m/p} = {\mathbf y}_n + \frac{m}{p} \, \Delta t \, {\mathbf y}'(t_n) + 
\frac{m^2}{p^2} \, \frac{\Delta t^2}{2} \, {\mathbf y}''(t_n) + 
\frac{m^3}{p^3} \, \frac{\Delta t^3}{6} \, {\mathbf y}'''(t_n) + {\mathcal O}(\Delta t^4)\,, \quad m = -3, \dots, p-1\,.
\end{equation}
Next, we insert (\ref{lts-ab4-1}) into (\ref{lts_ab}) with $k=4$ and $\Delta \tau = \Delta t / p$, replace ${\mathbf y}_{n-1}$,  ${\mathbf y}_{n-2}$ and 
${\mathbf y}_{n-3}$ by their respective Taylor expansions, use (\ref{prop_alpha}) and (\ref{prop_beta}) with $k=4$, and find after further simplifications
$$
{\mathbf y}_{n+1} = {\mathbf y}_n + \Delta t \, {\mathbf y}'(t_n) + \frac{\Delta t^2}{2}\,{\mathbf y}''(t_n) 
+ \frac{\Delta t^3}{6}\,{\mathbf y}'''(t_n) + \frac{\Delta t^4}{24}\,{\mathbf y}''''(t_n) + {\mathcal O}(\Delta t^5)\,,
$$
which completes the proof for the LTS-AB$4$($p$) scheme.
\end{proof}

\section{Numerical experiments}
Here we present numerical experiments that illustrate the stability properties of the above LTS methods, validate their expected 
order of convergence and demonstrate their usefulness in the presence of complex geometry. First, we consider a simple 
one-dimensional test problem to study the stability of the different LTS schemes presented above and to show that they 
yield the expected overall rate of convergence when combined with a spatial finite element discretization of comparable 
accuracy, independently of the number of local time-steps $p$ used in the fine region. Then, we illustrate the versatility of our LTS schemes 
by simulating the propagation of a plane wave as it impinges on a roof mounted antenna.

\subsection{Stability}
We consider the one-dimensional homogeneous damped wave equation (\ref{model_eq_1}) with constant wave speed $c = 1$ and 
damping coefficient $\sigma = 0.1$ on the interval $\Omega = [0\, , \,6]$. Next, we divide $\Omega$ into three
equal parts.  The left and right intervals, $[0\, , \,2]$ and $[4\, , \,6]$, respectively, are discretized with
an equidistant mesh of size $h^{\mbox{\scriptsize coarse}}$, whereas the interval $\Omega_f = [2\, , \,4]$ is discretized with an
equidistant mesh of size $h^{\mbox{\scriptsize fine}} = h^{\mbox{\scriptsize coarse}} / p$ - see Fig. \ref{1D_domain}. Hence, the two outer intervals 
correspond to the coarse region whereas the inner interval $[2\, , \,4]$ to the refined region.

For every time-step $\Delta t$, we shall take $p$ steps of size $\Delta \tau = \Delta t / p$ inside the refined region $\Omega_f$. In the 
absence of local refinement, i.e., $p=1$, the mesh is equidistant throughout $\Omega$.  Then, the LTS-AB$k$($p$) algorithm 
reduces to the standard $k$-step Adams-Bashforth method and we denote by $\Delta t_{ABk}$ the largest time-step allowed. 
For $p \geq 2$, we let $\Delta t_{p}$ denote the maximal time-step permitted in the LTS-AB$k$($p$) algorithm; clearly, we always have 
$\Delta t_p \leq \Delta t_{ABk}$. When $\Delta t_{p} = \Delta t_{ABk}$, the LTS algorithm imposes no further restriction on $\Delta t$; we then call 
the CFL condition of the LTS scheme optimal. We shall now evaluate numerically the CFL condition of the various LTS schemes from Section 3. 
To do so,  we proceed as follows:
\begin{enumerate}
 \item Set $\Delta t_{ABk}$ to the maximal $\Delta t$ allowed for the equidistant mesh of mesh size $h^{\mbox{\scriptsize coarse}}$;
 \item refine the equidistant mesh $p$ times inside $\Omega_f$;
 \item determine the maximal time-step $\Delta t_{p}$ allowed by the LTS-AB$k$($p$) method on the locally refined mesh and compare $\Delta t_{p}$ with $\Delta t_{ABk}$.
\end{enumerate}

\begin{figure}[t]
\centerline{\includegraphics[scale=0.8]{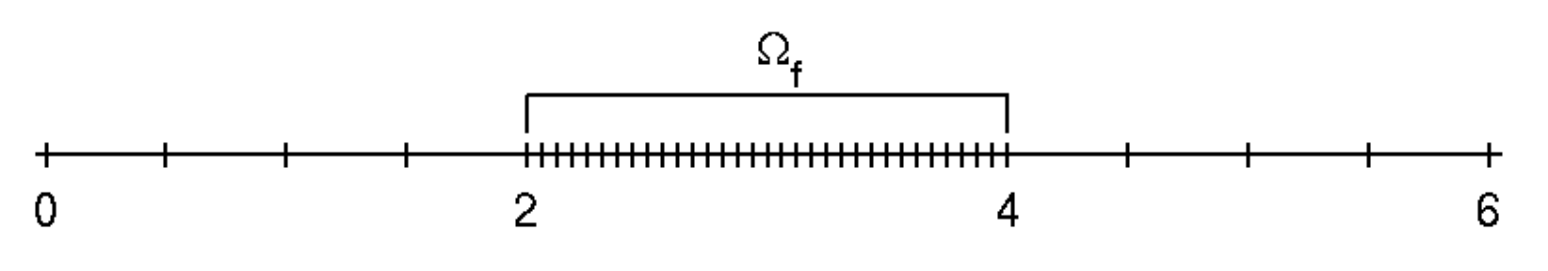}}
\caption{One-dimensional example: the computational domain $\Omega = [0, 6]$ with the refined region $\Omega_f = [2, 4]$.} 
\label{1D_domain}
\end{figure}

First, we consider a ${\cal P}^1$ continuous FE discretization with mass lumping and combine it with the second-order 
Adams-Bashforth scheme. We choose $h^{\mbox{\scriptsize coarse}} = 0.1$ and rewrite the two-step Adams-Bashforth method as the one-step method
$$
\left ( \begin{array}{l} {\mathbf y}_{n+1} \\ {\mathbf y}_{n} \end{array} \right ) = 
{\mathbf C}_{AB2} \left ( \begin{array}{l} {\mathbf y}_{n} \\ {\mathbf y}_{n-1} \end{array} \right ) \,, \quad 
{\mathbf C}_{AB2} = \left ( \begin{array}{cc}
{\mathbf I} + \frac{3}{2} \Delta t {\mathbf B} & -\frac{\Delta t}{2} {\mathbf B} \\[5pt]
{\mathbf I} & {\mathbf 0}
\end{array} \right )\,.
$$
It is stable if $\rho(C_{AB2}) \leq 1$, where $\rho(C_{AB2})$ denotes the spectral radius of the matrix $C_{AB2}$ \cite{HNW00}. 
Progressively increasing $\Delta t$ while monitoring $\rho(C_{AB2})$, we find that the maximal time-step allowed is $\Delta t_{AB2} = 0.0106$ for $p=1$. 
Next, we refine by a factor $p = 2$ those 
elements that lie inside the interval $[2,4]$, that is $h^{\mbox{\scriptsize fine}} = 0.05$, and set to one all corresponding entries in $P$. 
Hence, for every time-step $\Delta t$, we shall take two steps of size $\Delta \tau = \Delta t / 2$ inside the refined region with the second-order 
time-stepping scheme LTS-AB$2$($2$). To determine the stability of the LTS-AB$2$($2$) scheme, we also rewrite it as a one-step method:
 $$
\left ( \begin{array}{l} {\mathbf y}_{n+1} \\ \widetilde {\mathbf y}_{1/2} \\ {\mathbf y}_{n} \end{array} \right ) = 
{\mathbf C}_{LTS-AB2(2)} 
\left ( \begin{array}{l} {\mathbf y}_{n} \\ \widetilde{\mathbf y}_{-1/2} \\ {\mathbf y}_{n-1} \end{array} \right ) \,, \quad 
{\mathbf C}_{LTS-AB2(2)} = \left ( \begin{array}{ccc}
{\mathbf C}_{11} & {\mathbf C}_{12} &  {\mathbf C}_{13} \\
{\mathbf C}_{21} & {\mathbf C}_{22} &  {\mathbf C}_{23} \\
{\mathbf I} & {\mathbf 0} & {\mathbf 0}
\end{array} \right )\,,
$$
with
\begin{equation*} \begin{split}
{\mathbf C}_{11} & = {\mathbf I} + \frac{3}{2} \, \Delta \, {\mathbf B} - \frac{1}{4} \, \Delta t \, \mathbf {BP} 
+ \frac{3}{8} \left ( \frac{\Delta t}{2} \right )^2 (\mathbf {BP})^2 + \frac{15}{8}  \left ( \frac{\Delta t}{2} \right )^2 \mathbf {BPB}\,, \\
{\mathbf C}_{12} & = - \frac{1}{4} \, \Delta t \, \mathbf {BP}  - \frac{3}{4} \left ( \frac{\Delta t}{2} \right )^2 (\mathbf {BP})^2 \,, \\
{\mathbf C}_{13} & = - \frac{1}{2} \, \Delta \, {\mathbf B} + \frac{1}{2} \, \Delta t \, \mathbf {BP} 
+ \frac{3}{8} \left ( \frac{\Delta t}{2} \right )^2 (\mathbf {BP})^2 - \frac{3}{8}  \left ( \frac{\Delta t}{2} \right )^2 \mathbf {BPB}\,, \\
{\mathbf C}_{21} & = {\mathbf I} + \frac{5}{8} \, \Delta \, {\mathbf B} + \frac{1}{8} \, \Delta t \, \mathbf {BP} \,, \quad 
{\mathbf C}_{22} = - \frac{1}{4} \, \Delta t \, \mathbf {BP} \,, \quad 
{\mathbf C}_{23} = - \frac{1}{8} \, \Delta \, {\mathbf B} + \frac{1}{8} \, \Delta t \, \mathbf {BP} \,.
\end{split} \end{equation*}
The LTS-AB$2$($2$) scheme is stable if $\rho({\mathbf C}_{LTS-AB2(2)}) \leq 1$, where $\rho({\mathbf C}_{LTS-AB2(2)})$ denotes 
the spectral radius of the matrix ${\mathbf C}_{LTS-AB2(2)}$. 

To determine the range of values $\Delta t$ for which the LTS-AB$2$($2$) scheme is stable, we compute the spectral radius of 
${\mathbf C}_{LTS-AB2(2)}$ for varying $\Delta t$. As shown in the right frame of Fig.~\ref{plot_eigs_AB2}, the spectral radius transgresses the 
stability threshold at one for a time-step about 80\% of $\Delta t_{AB2}$. Thus, for all time-steps $\Delta t \leq 0.8 \, \Delta t_{AB2}$, the LTS-AB$2$($2$) 
scheme is stable; still, its CFL condition is not optimal.

\begin{figure}[t]
\centering
\begin{tabular}{cc}
\includegraphics[scale=0.25]{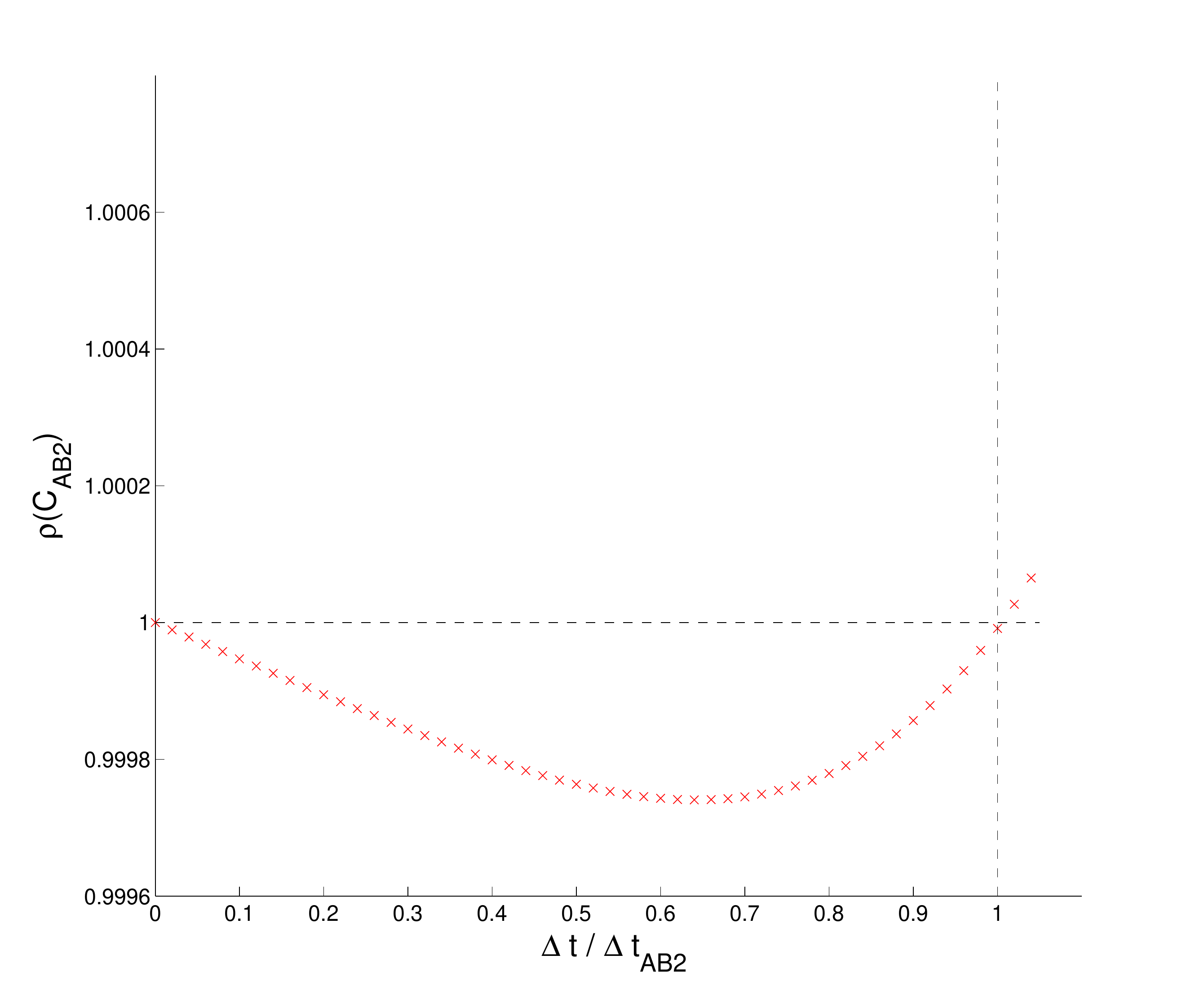}  & \includegraphics[scale=0.25]{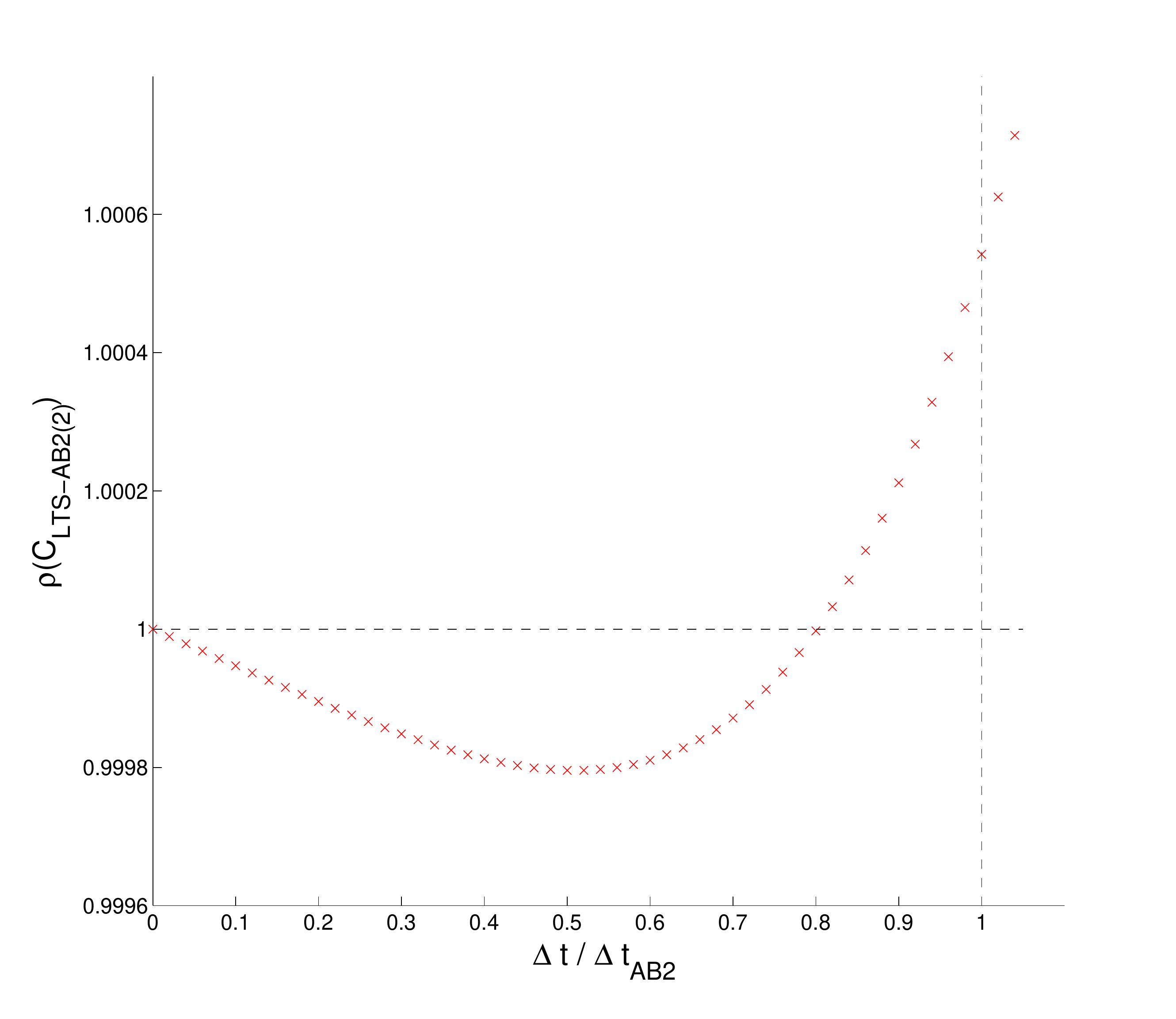}  
\end{tabular}
\caption{The classical AB$2$ and the LTS-AB$2$($2$) schemes combined with ${\cal P}^1$ continuous FE: the spectral radius of ${\mathbf C}_{AB2}$ (left) 
and ${\mathbf C}_{LTS-AB2(2)}$ (right) is shown for varying $\Delta t / \Delta t_{AB2}$.} 
\label{plot_eigs_AB2}
\end{figure}

\begin{figure}[t]
\centering
\begin{tabular}{cc}
\includegraphics[scale=0.25]{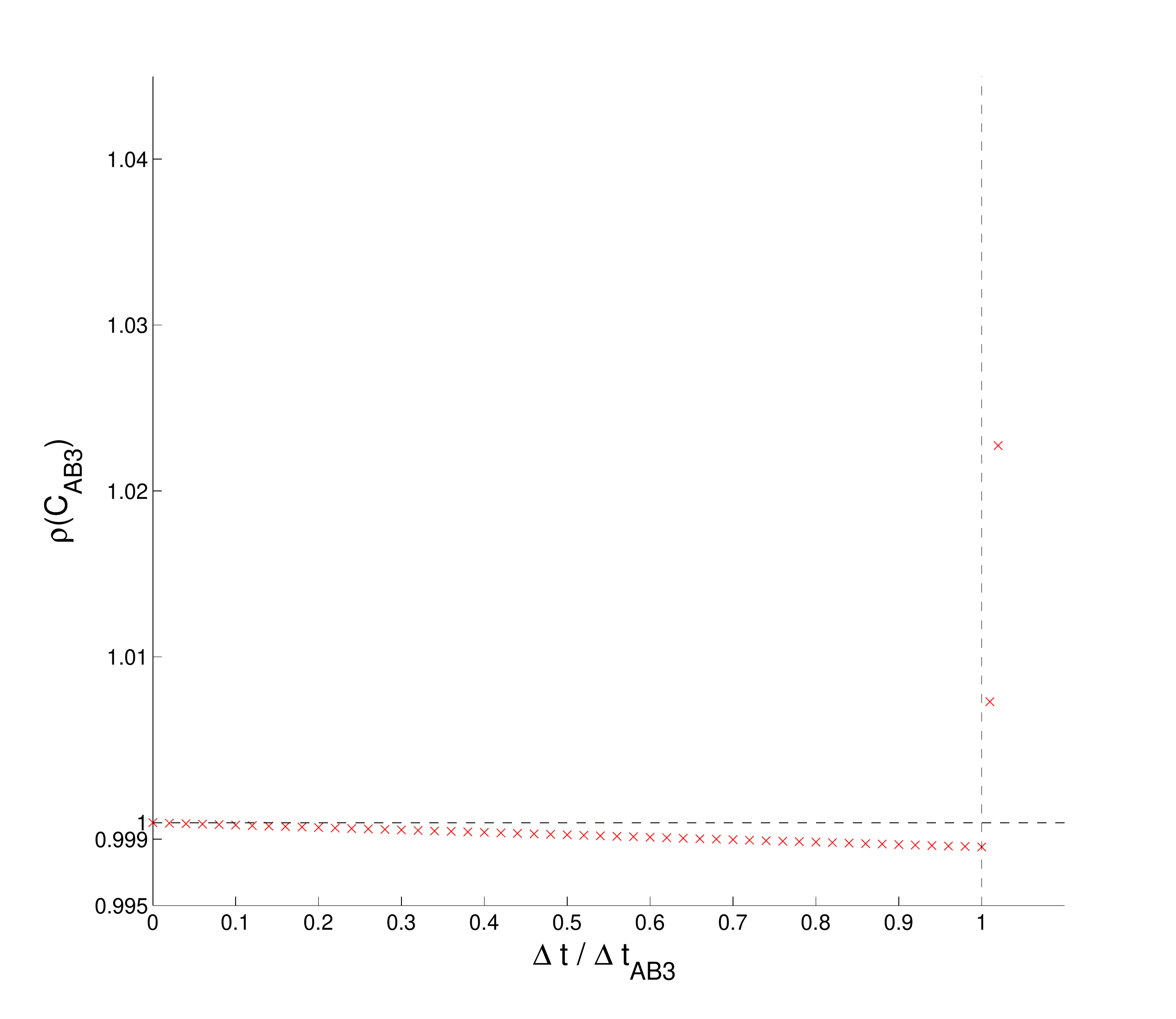}  & \includegraphics[scale=0.25]{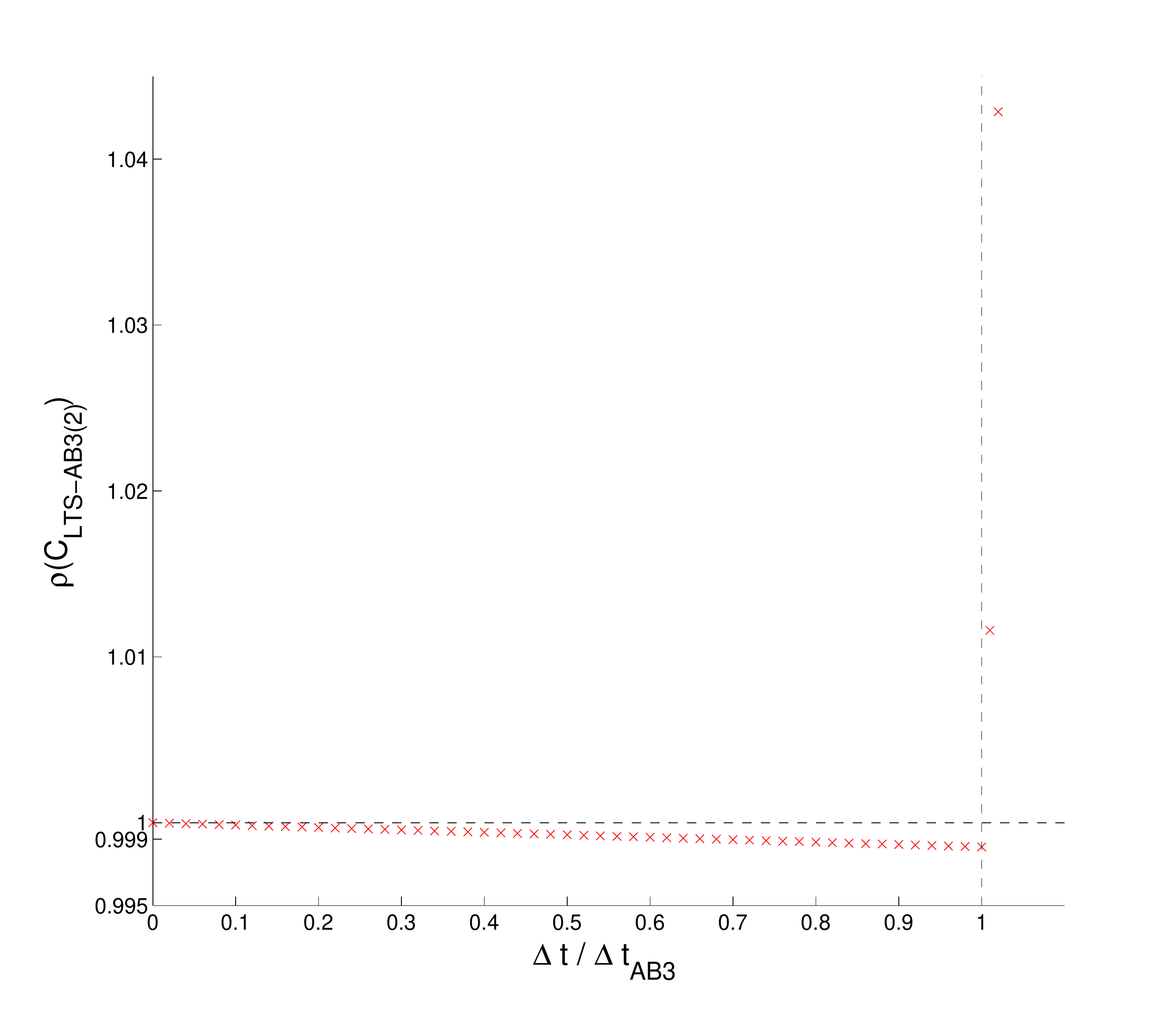}  
\end{tabular}
\caption{The classical AB$3$ and the LTS-AB$3$($2$) schemes combined with ${\cal P}^2$ continuous FE: the spectral radius of ${\mathbf C}_{AB3}$ (left) 
and ${\mathbf C}_{LTS-AB3(2)}$ (right) is shown for varying $\Delta t / \Delta t_{AB3}$.} 
\label{plot_eigs_AB3}
\end{figure}

Next, we consider the third-order Adams-Bashforth scheme and combine it with a ${\cal P}^2$ continuous FE discretization 
with mass lumping. We choose $h^{\mbox{\scriptsize coarse}} = 0.2$, which yields the maximal time-step 
$\Delta t_{AB3} = 0.029$ for $p=1$. Again, we refine 
by a factor $p = 2$ those elements that lie inside the interval $[2,4]$, that is $h^{\mbox{\scriptsize fine}} = 0.1$. 
Hence, for every time-step $\Delta t$, we shall take two steps of size $\Delta \tau = \Delta t / 2$ in the refined region with the third-order LTS-AB$3$($2$) scheme. 

To study its stability properties, we again reformulate both the classical AB$3$ and the LTS-AB$3$($2$) schemes as 
one-step methods:
$$
\left ( \begin{array}{l} {\mathbf y}_{n+1} \\ {\mathbf y}_{n} \\ {\mathbf y}_{n-1} \end{array} \right ) = 
{\mathbf C}_{AB3} \left ( \begin{array}{l} {\mathbf y}_{n} \\ {\mathbf y}_{n-1} \\ {\mathbf y}_{n-2} \end{array} \right ) 
\,, \quad
\left ( \begin{array}{l} 
{\mathbf y}_{n+1} \\ \widetilde {\mathbf y}_{1/2} \\ {\mathbf y}_{n} \\ {\mathbf y}_{n-1}
\end{array} \right ) = 
{\mathbf C}_{LTS-AB3(2)} 
\left ( \begin{array}{l} 
{\mathbf y}_{n} \\ \widetilde{\mathbf y}_{-1/2} \\ {\mathbf y}_{n-1} \\ {\mathbf y}_{n-2}
\end{array} \right ) \,
$$
and compute the spectral radius of ${\mathbf C}_{AB3}$ and ${\mathbf C}_{LTS-AB3(2)}$ for varying $\Delta t$. 
In the right frame of Fig.~\ref{plot_eigs_AB3}, we observe that the spectral radius of ${\mathbf C}_{LTS-AB3(2)}$ lies below one 
for all time-steps $\Delta t \leq \Delta t_{AB3}$. Hence, the LTS-AB$3$($2$) scheme is stable up to the maximal time-step allowed by 
the standard third-order Adams-Bashforth method on an equidistant mesh; therefore its CFL condition is optimal.

\begin{figure}[t]
\centering
\begin{tabular}{cc}
\includegraphics[scale=0.24]{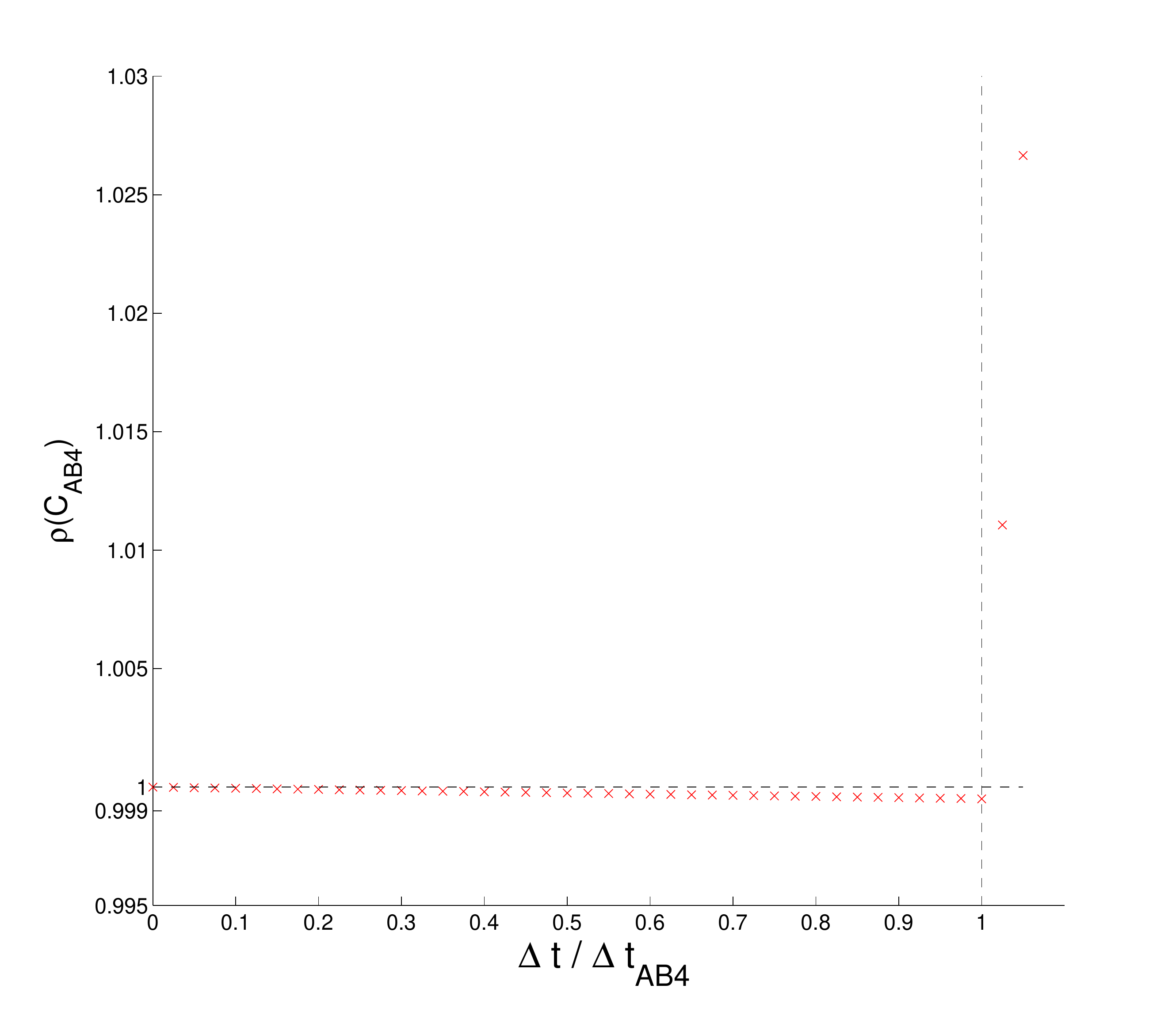}  & \includegraphics[scale=0.24]{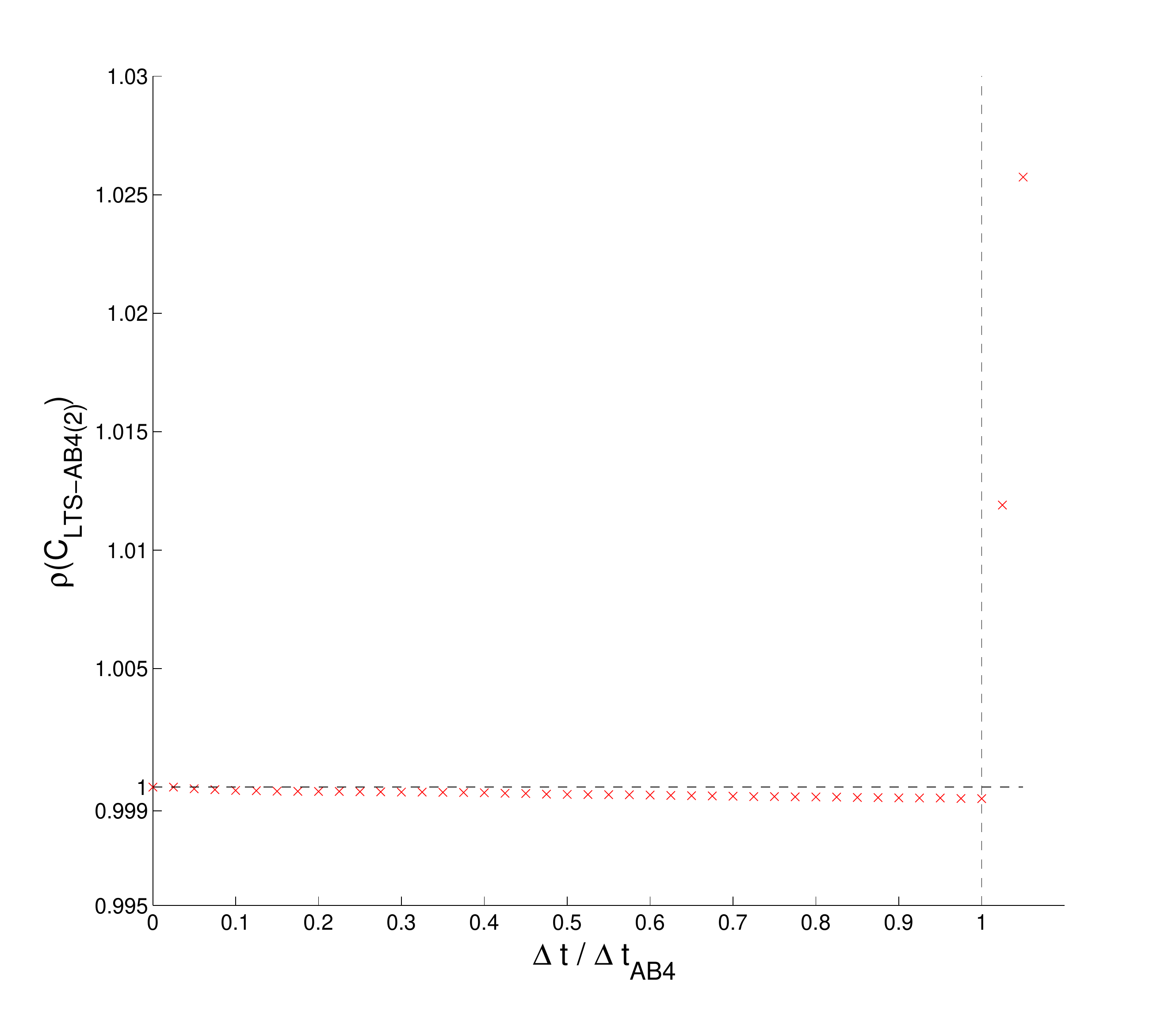}  
\end{tabular}
\caption{The classical AB$4$ and the LTS-AB$4$($2$) schemes combined with ${\cal P}^3$ continuous FE: the spectral radius of ${\mathbf C}_{AB4}$ (left) 
and ${\mathbf C}_{LTS-AB4(2)}$ (right) is shown for varying $\Delta t / \Delta t_{AB4}$.} 
\label{plot_eigs_AB4}
\end{figure}

\begin{figure}[t]
\centering
\begin{tabular}{cc}
\includegraphics[scale=0.24]{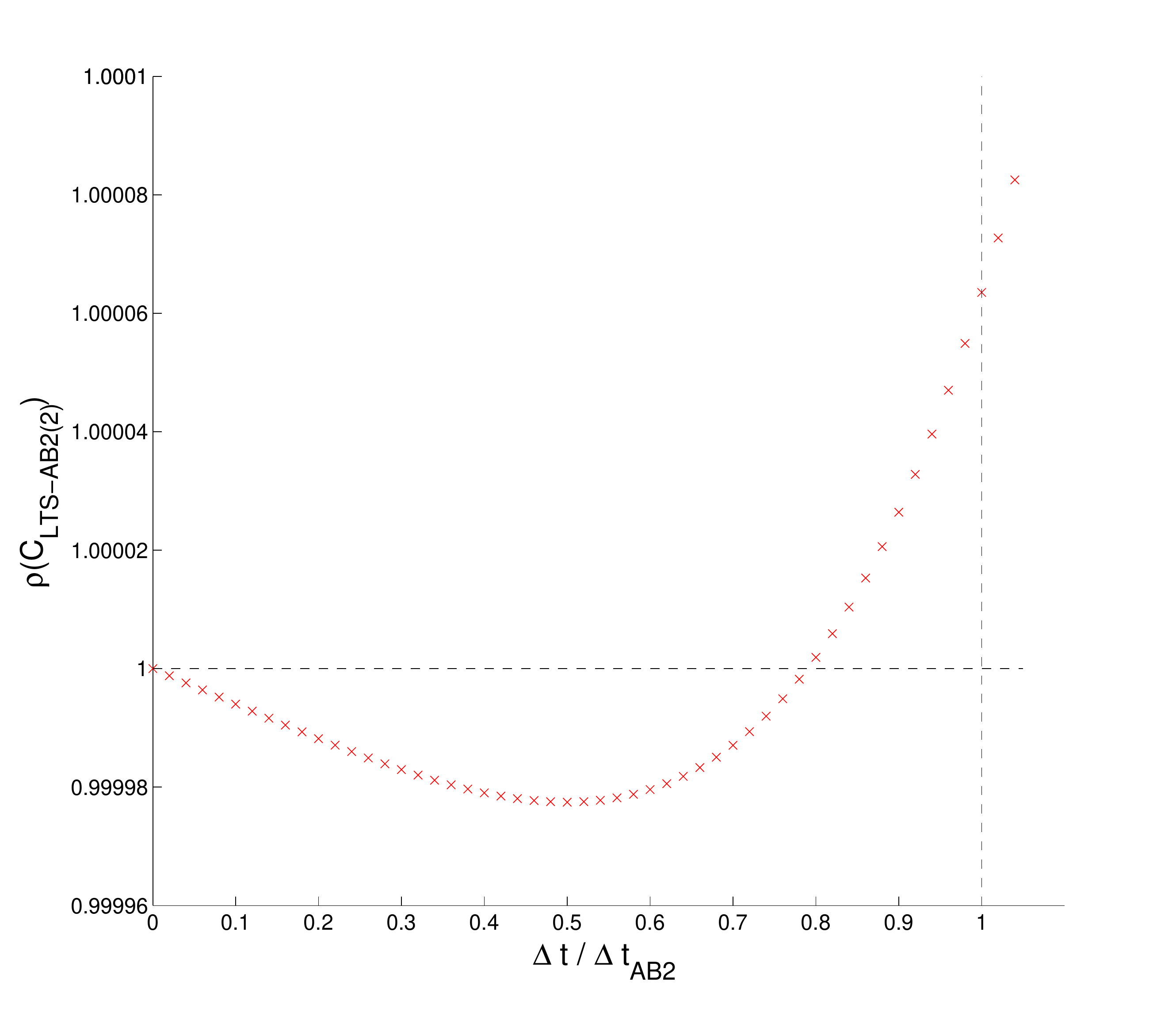}  & \includegraphics[scale=0.24]{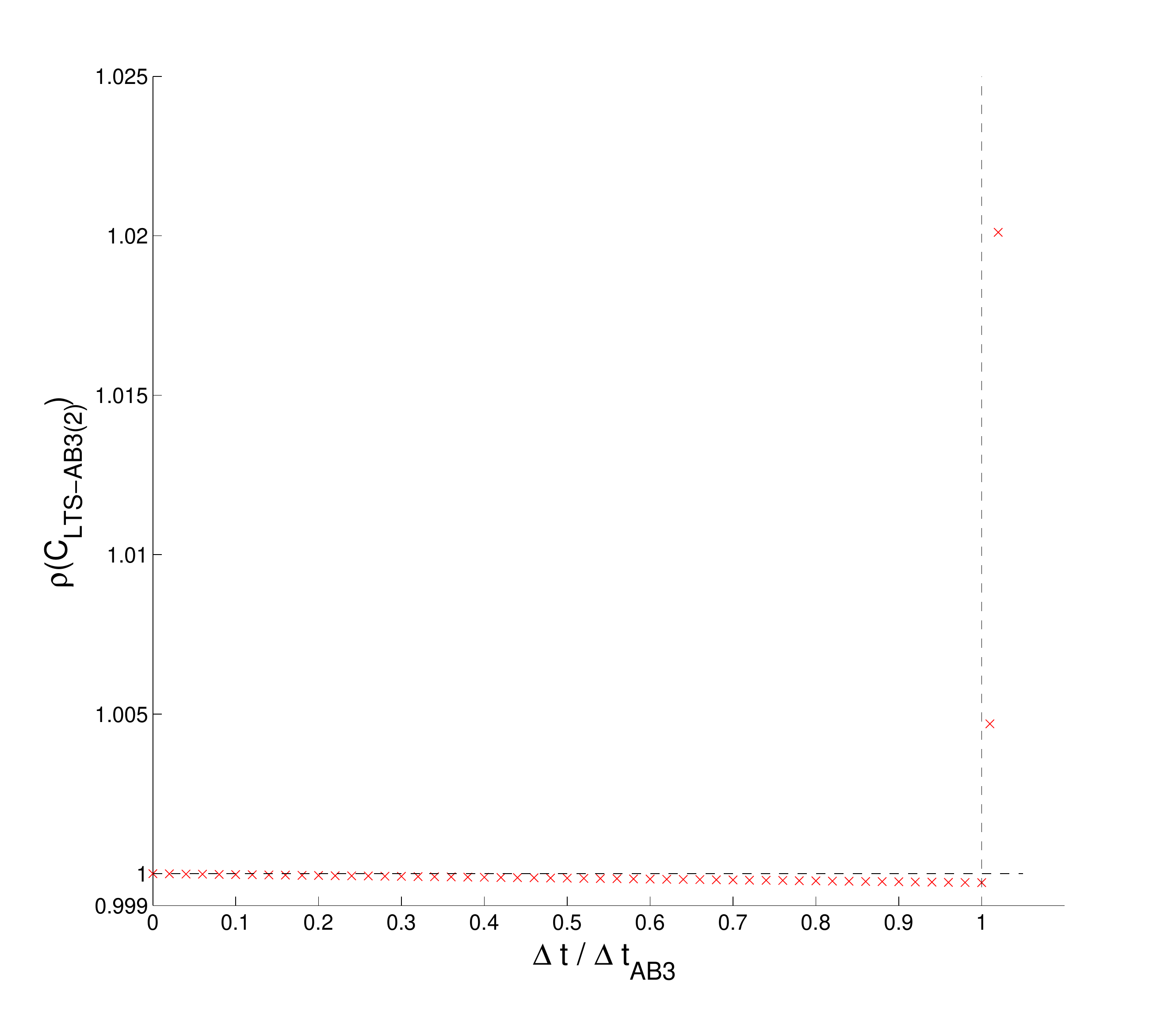}  
\end{tabular}
\caption{The $k$-th order LTS-AB$k$($2$) scheme combined with IP-DG ${\cal P}^{k-1}$-elements: the spectral radius of ${\mathbf C}_{LTS-ABk(2)}$ with 
$k=2$ (left) and $k=3$ (right) is shown for varying $\Delta t / \Delta t_{ABk}$.} 
\label{plot_eigs_AB_IPDG}
\end{figure}

Finally, to study the stability of the fourth-order LTS scheme, we consider a ${\cal P}^3$ continuous FE discretization with mass lumping. Again, 
we choose $h^{\mbox{\scriptsize coarse}} = 0.2$, which now yields the maximal time-step $\Delta t_{AB4} = 0.0099$ for 
$p=1$ and refine by a factor $p = 2$ those elements that lie inside the interval $[2,4]$. Hence, for every time-step $\Delta t$, we 
shall use the fourth-order time-stepping scheme LTS-AB$4$($2$) with $\Delta \tau = \Delta t / 2$ in the refined region. After refomulating the 
AB$4$ and the LTS-AB$4$($2$) schemes as one-step methods
$$
\left ( \begin{array}{l} {\mathbf y}_{n+1} \\ {\mathbf y}_{n} \\ {\mathbf y}_{n-1} \\ {\mathbf y}_{n-2} \end{array} \right ) = 
{\mathbf C}_{AB4} \left ( \begin{array}{l} {\mathbf y}_{n} \\ {\mathbf y}_{n-1} \\ {\mathbf y}_{n-2} \\ {\mathbf y}_{n-3} \end{array} \right ) 
\,, \quad
\left ( \begin{array}{l} 
{\mathbf y}_{n+1} \\ \widetilde {\mathbf y}_{1/2} \\ {\mathbf y}_{n} \\ \widetilde{\mathbf y}_{-1/2} \\ {\mathbf y}_{n-1} \\ \widetilde{\mathbf y}_{-3/2}
\end{array} \right ) = 
{\mathbf C}_{LTS-AB4(2)} 
\left ( \begin{array}{l} 
{\mathbf y}_{n} \\ \widetilde{\mathbf y}_{-1/2} \\ {\mathbf y}_{n-1} \\ \widetilde{\mathbf y}_{-3/2} \\ {\mathbf y}_{n-2} \\ {\mathbf y}_{n-3}
\end{array} \right ) \,,
$$
we compute the spectral radius of ${\mathbf C}_{AB4}$ and ${\mathbf C}_{LTS-AB4(2)}$ for varying $\Delta t / \Delta t_{AB4}$. As shown in the right frame of 
Fig.~\ref{plot_eigs_AB4}, the spectral radius of ${\mathbf C}_{LTS-AB4(2)}$ lies below one for all time-steps $\Delta t \leq \Delta t_{AB4}$. Thus, 
the LTS-AB$4$($2$) scheme is stable up to the optimal time-step.

So far, we have restricted the detailed numerical stability study of the LTS schemes to standard continuous FE 
discretizations, where the accuracy of the spatial discretization matches that of the time discretization.
We obtain similar stability results, when the LTS-AB$k$ schemes are combined with a $k$-th order IP-DG or nodal DG discretization in space. For instance, as shown 
in the left frame of Fig.~ \ref{plot_eigs_AB_IPDG}, the largest time-step allowed by the LTS-AB$2$(2) scheme when combined with IP-DG ${\cal P}^{1}$-elements 
($\alpha = 5$ in (\ref{param_alpha})) again is 
about 80\% of $\Delta t_{AB2}$. Similarly, the spectral radius of ${\mathbf C}_{LTS-AB3(2)}$, shown in the right frame of Fig.~\ref{plot_eigs_AB_IPDG} 
for varying $\Delta t / \Delta t_{AB3}$, reveals that the LTS-AB$3$(2) scheme, when combined with IP-DG ${\cal P}^{2}$-elements (with $\alpha = 12$), is again stable for the optimal 
time-step $\Delta t_{AB3}$.

Finally to illustrate the effect of $\sigma$ on the stability, we display in Table \ref{table_sigma}
the time-step ratio $\Delta t_2 / \Delta t_{ABk}$ for varying $\sigma$, 
either with a conforming or an IP-DG finite element discretization for $k=2, 3, 4$. We observe that
the LTS-AB$k$ methods with $k \geq 3$ yield the optimal CFL condition independently of $\sigma$. In fact,  
they can even be used for $\sigma = 0$, that is in the undamped regime. In contrast, the LTS-AB$2$ scheme
typically yields only about 80\% of the optimal the CFL condition and can only be used if $\sigma >0$.  

\begin{table}[t!]
\begin{center}
\begin{tabular}{c||c|c||c|c||c|c} \hline
& \multicolumn{2}{c||}{LTS-AB2(2)} & \multicolumn{2}{c||}{LTS-AB3(2)} & \multicolumn{2}{c}{LTS-AB4(2)}\\
$\sigma$ & cont. FE & IP-DG & cont. FE & IP-DG & cont. FE & IP-DG \\ \hline
0 & - & - & 1.0 & 1.0 & 1.0 & 1.0 \\ \hline
0.001 & 0.79 & 0.8 & 1.0 & 1.0 & 1.0 & 1.0 \\ \hline
0.1 & 0.8 & 0.81 & 1.0 & 1.0 & 1.0 & 1.0 \\ \hline
1 & 0.8 & 0.81 & 1.0 & 1.0 & 1.0 & 1.0 \\ \hline
10 & 0.86 & 0.87 & 1.0 & 1.0 & 1.0 & 1.0 \\ \hline
\end{tabular} 
\caption{The LTS-AB$k$(2) scheme for $k=2, 3, 4$ combined with ${\mathcal P}^{k-1}$ continuous FE or IP-DG ${\mathcal P}^{k-1}$ elements: the ratio 
         $\Delta t_2 / \Delta t_{ABk}$ is shown for varying $\sigma$.}
\label{table_sigma}
\end{center}
\end{table}

\begin{remark}
In summary, our numerical experiments for
       $h=0.1, 0.2$, $1\leq p\leq 13$, and $0\leq \sigma\leq 10$ 
show for a spatial discretization with standard continuous, IP-DG or nodal DG finite elements:
\begin{itemize}
 \item the maximal time-step $\Delta t_{p}$ allowed by the LTS-AB\,$2$($p$) scheme is about 80 \% of the optimal time-step $\Delta t_{AB2}$ for $\sigma > 0$;
 \item the CFL stability condition of the LTS-AB\,$3$($p$) and LTS-AB\,$4$($p$) schemes is optimal.
\end{itemize}
These numerical results suggest that the above two claims probably also
hold for all other values of $p$, $h$ or $\sigma>0$.
\end{remark}

\subsection{Convergence}
We consider the one-dimensional homogeneous model problems (\ref{model_eq_1}) and (\ref{model_eq_2}) with constant wave speed $c = 1$ and 
damping coefficient $\sigma = 0.1$ on the interval $\Omega = [0\, , \,6]$. The initial conditions are chosen to yield 
the exact solution
\begin{equation*} \begin{split}
u(x, t) & = \frac{2 e^{-\frac{\sigma t}{2}}}{\sqrt{4\pi^2 - \sigma^2}} \sin(\pi x) 
\sin\left ( \frac{t}{2} \sqrt{4\pi^2 - \sigma^2}\right ) \,, \\
v(x, t) & =  \frac{\partial u}{\partial t}(x, t) \,, \quad {\mathbf w} (x, t) = - \nabla u (x, t) \,.
\end{split} \end{equation*}
Again, we divide $\Omega$ into three equal parts. The left and right intervals, $[0, 2]$ and $[4, 6]$, 
respectively, are discretized with an equidistant mesh of size $h^{\mbox{\scriptsize coarse}}$, whereas the interval
$[2, 4]$ is discretized with an equidistant mesh of size 
$h^{\mbox{\scriptsize fine}} = h^{\mbox{\scriptsize coarse}} / p$. Hence, the two outer
intervals correspond to the coarse region and the inner interval $[2\, , \,4]$ to the refined region. 
The first $k-1$ time-steps of each LTS-AB$k$($p$) scheme are initialized by using the exact solution.

\begin{figure}[t!]
\centering
\begin{tabular}{cc}
\includegraphics[scale=0.2]{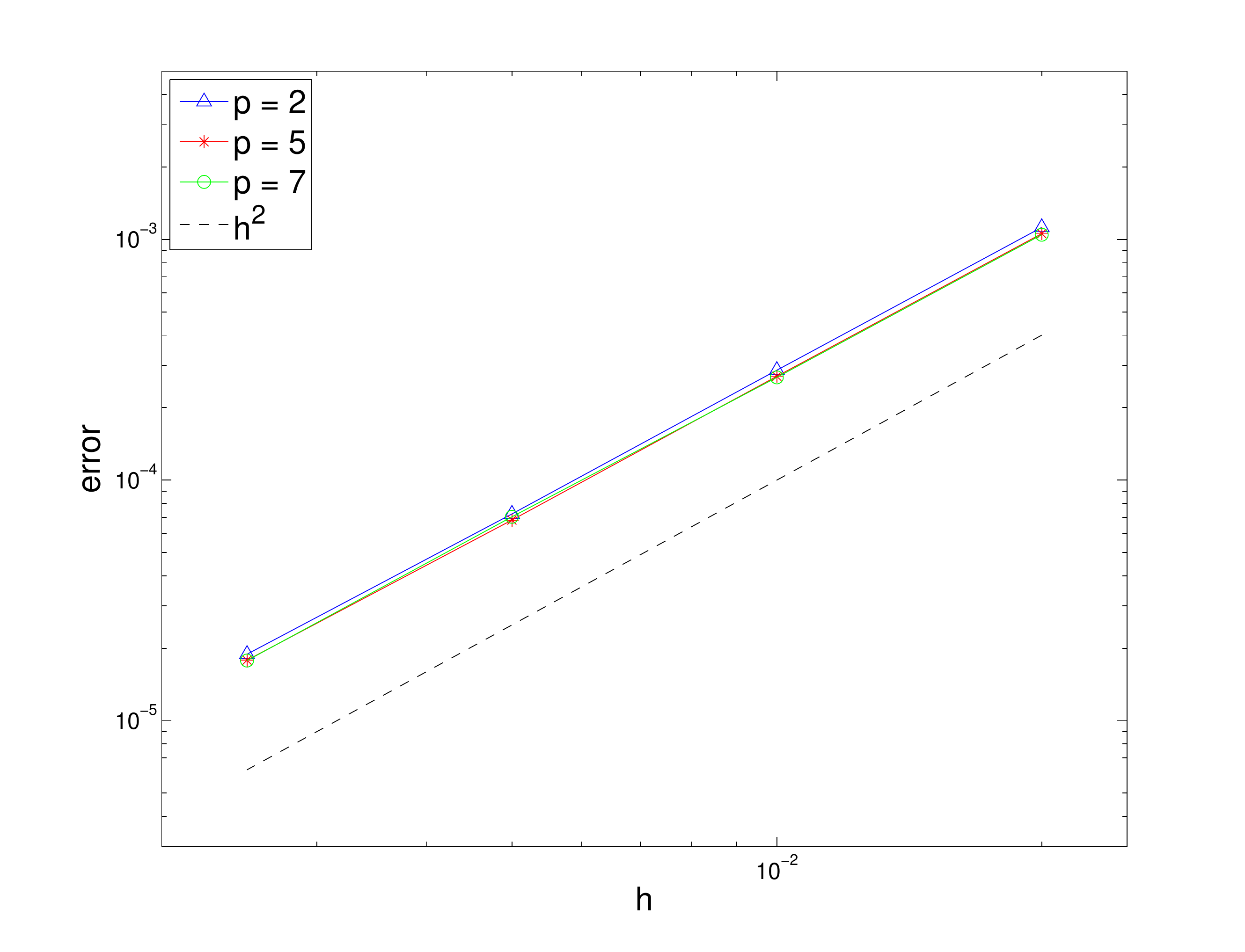}  & 
\includegraphics[scale=0.2]{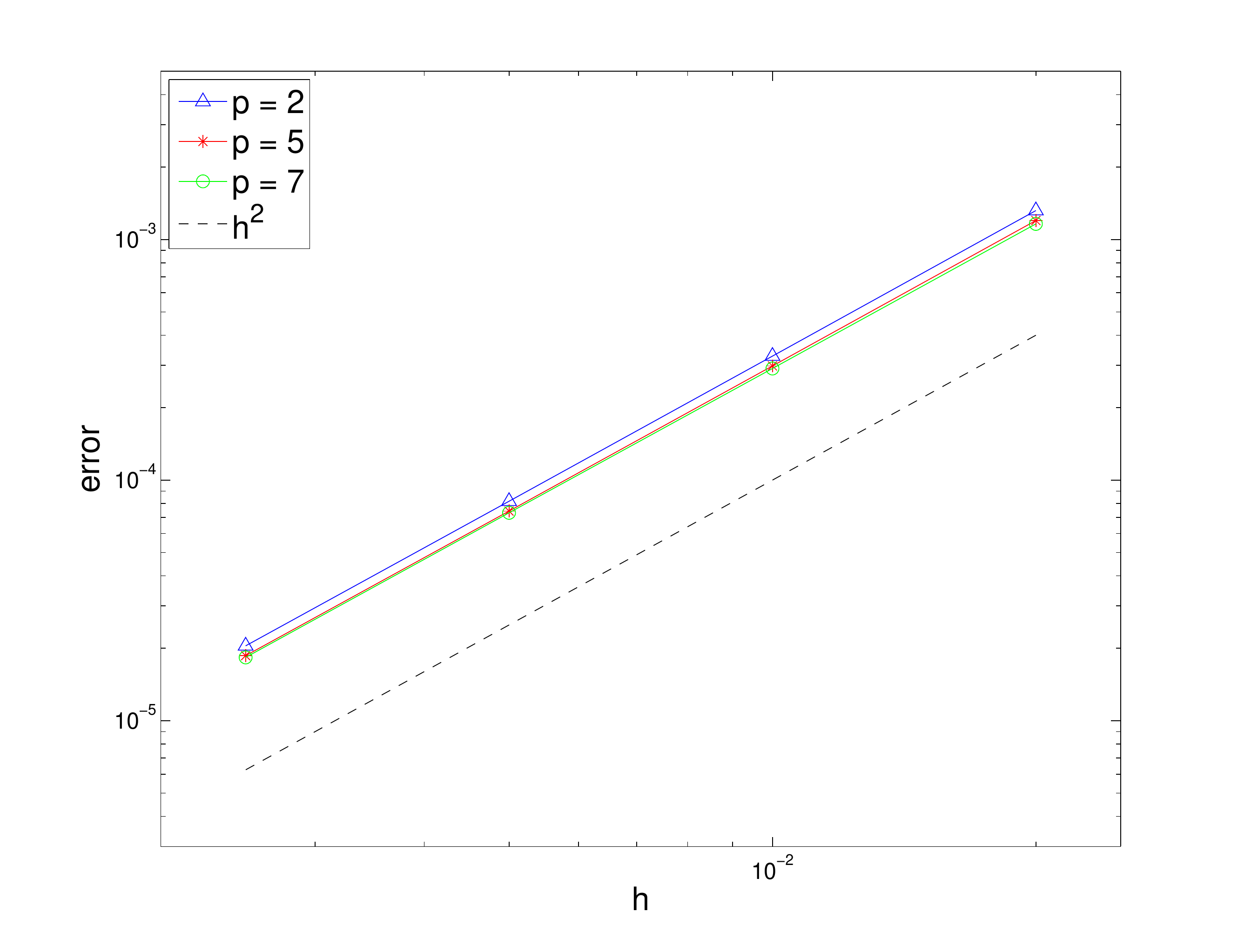}  \\
{\footnotesize (a) continuous FE ($h$ = 0.02, 0.01, 0.005, 0.0025)} & {\footnotesize (b) IP-DG ($h$ = 0.02, 0.01, 0.005, 0.0025)}
\end{tabular}
\begin{tabular}{c}
\includegraphics[scale=0.2]{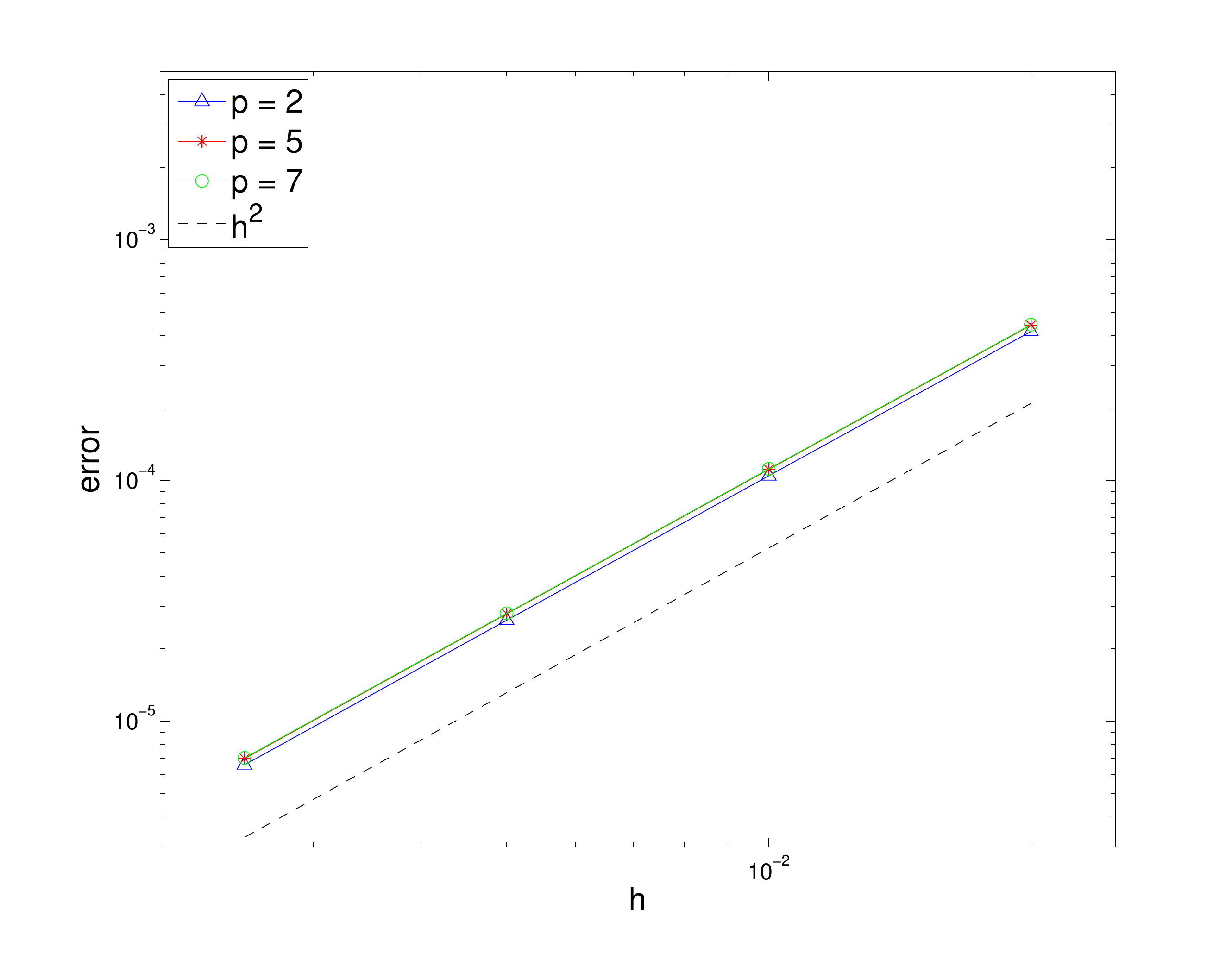} \\ {\footnotesize (c) nodal DG ($h$ = 0.02, 0.01, 0.005, 0.0025)}
\end{tabular}
\caption{LTS-AB$2$($p$) error vs.\ $h = h^{\mbox{\scriptsize coarse}}$ for ${\cal P}^1$ finite elements with $p=2, 5, 7$.} \label{1d_err_P1}
\end{figure}

First, we consider a ${\cal P}^1$ continuous FE discretization with mass lumping \cite{Coh02, CJT94} and a sequence of increasingly finer meshes. For every time-step $\Delta t$, we shall 
take $p \geq 2$ local steps of size $\Delta \tau = \Delta t / p$ in the refined region, with the second-order 
time-stepping scheme LTS-AB$2$($p$). According to our previous results on stability from Section 4.1, we set $\Delta t = 0.8 \cdot \Delta t_{AB2}$; note that 
$\Delta t_{AB2}$ depends on the mesh size. As we systematically reduce the global mesh size 
$h^{\mbox{\scriptsize coarse}}$, while simultaneously reducing $\Delta t$, we monitor the $L^2$ space-time error 
in the numerical solution $\| u(\cdot, T) - u^h(\cdot, T)\|_{L^2(\Omega)}$ at the final time $T = 10$. In frame (a) 
of Fig.~\ref{1d_err_P1}, the numerical error is shown vs.\ the mesh size $h = h^{\mbox{\scriptsize coarse}}$: 
regardless of the number of local time-steps $p =2$, 5 or 7, the numerical method converges with order two.

We now repeat the same experiment with the IP-DG ($\alpha = 5$ in (\ref{param_alpha})) and the nodal DG discretizations with ${\cal P}^1$-elements. As 
shown in frames (b) and (c) of Fig.~\ref{1d_err_P1}, the LTS-AB$2$($p$) method again yields overall second-order 
convergence independently of $p$.

Next, we consider the third-order LTS-AB$3$($p$) scheme and combine it with any one of the three FE discretizations with ${\cal P}^2$-elements. 
Thus, we expect all numerical schemes to exhibit overall third-order convergence with respect to the $L^2$-norm. We set $\Delta t = \Delta t_{AB3}$, the 
largest possible time-step allowed by the third-order Adams-Bashforth approach on an equidistant mesh with 
$h = h^{\mbox{\scriptsize coarse}}$. In Fig.~\ref{1d_err_P2} we display the space-time $L^2$-errors of the 
numerical solutions at $T = 10$ for a sequence of meshes and different values of $p$. The continuous FEM with mass lumping, 
the IP-DG method (with $\alpha = 12$) and the nodal DG discretization  all yield the expected third-order convergence.

Finally, to demonstrate the order of convergence of the fourth-order LTS-AB$4$($p$) scheme, we consider 
again the continuous FE or the two DG discretizations with ${\cal P}^3$-elements. Here, we set the penalty 
parameter $\alpha = 20$ for the IP-DG method and let $\Delta t =\Delta t_{AB4}$, the corresponding largest possible time-step allowed by 
the Adams-Bashforth approach of order four on an equidistant mesh with $h = h^{\mbox{\scriptsize coarse}}$. We 
monitor the $L^2$ space-time error in the numerical solution at $T = 10$ for the sequence of meshes and different values of $p$. Again, 
the numerical results shown in Fig.~\ref{1d_err_P3} corroborate the expected fourth-order convergence.

\begin{figure}[t!]
\centering
\begin{tabular}{cc}
\includegraphics[scale=0.2]{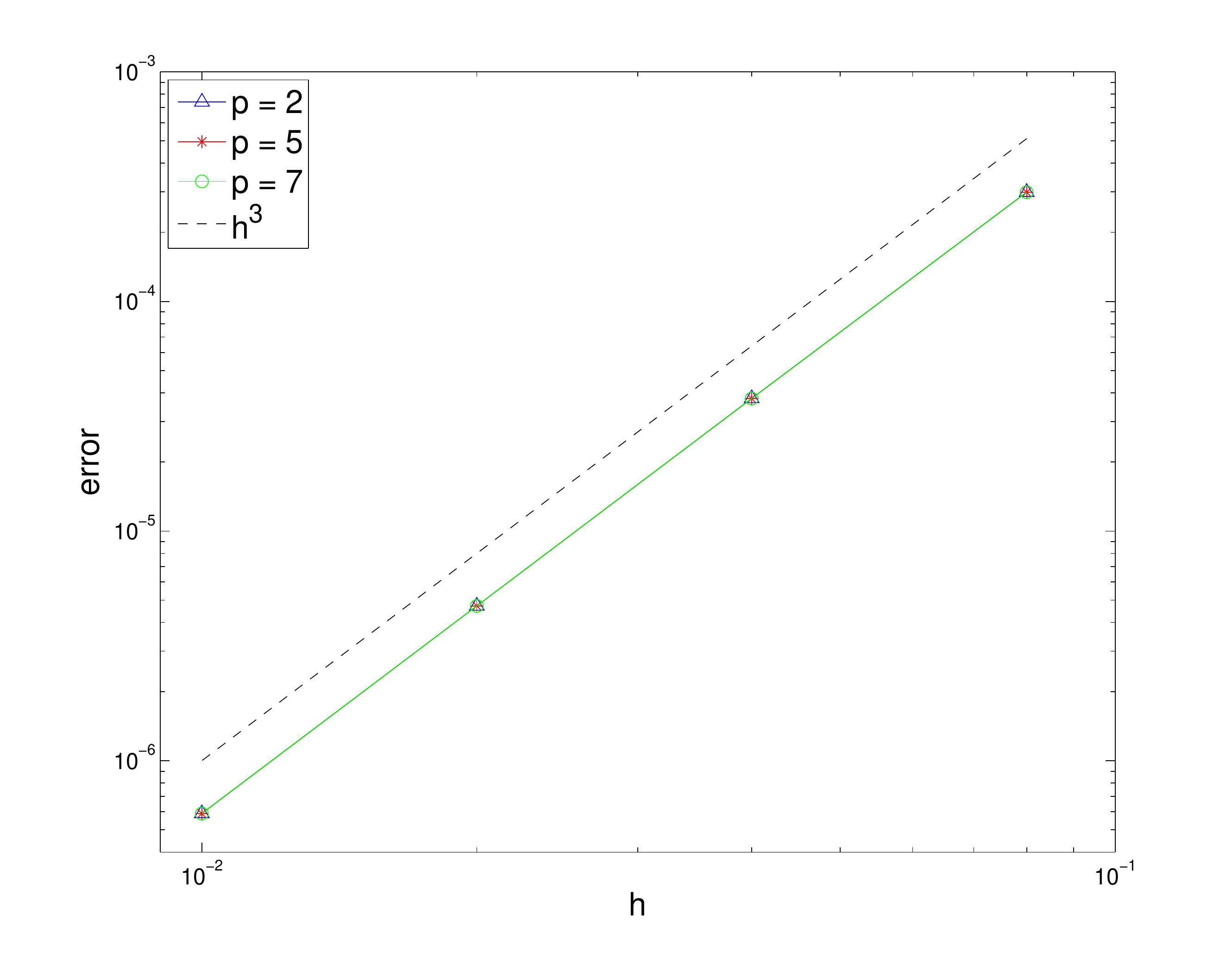}  & 
\includegraphics[scale=0.2]{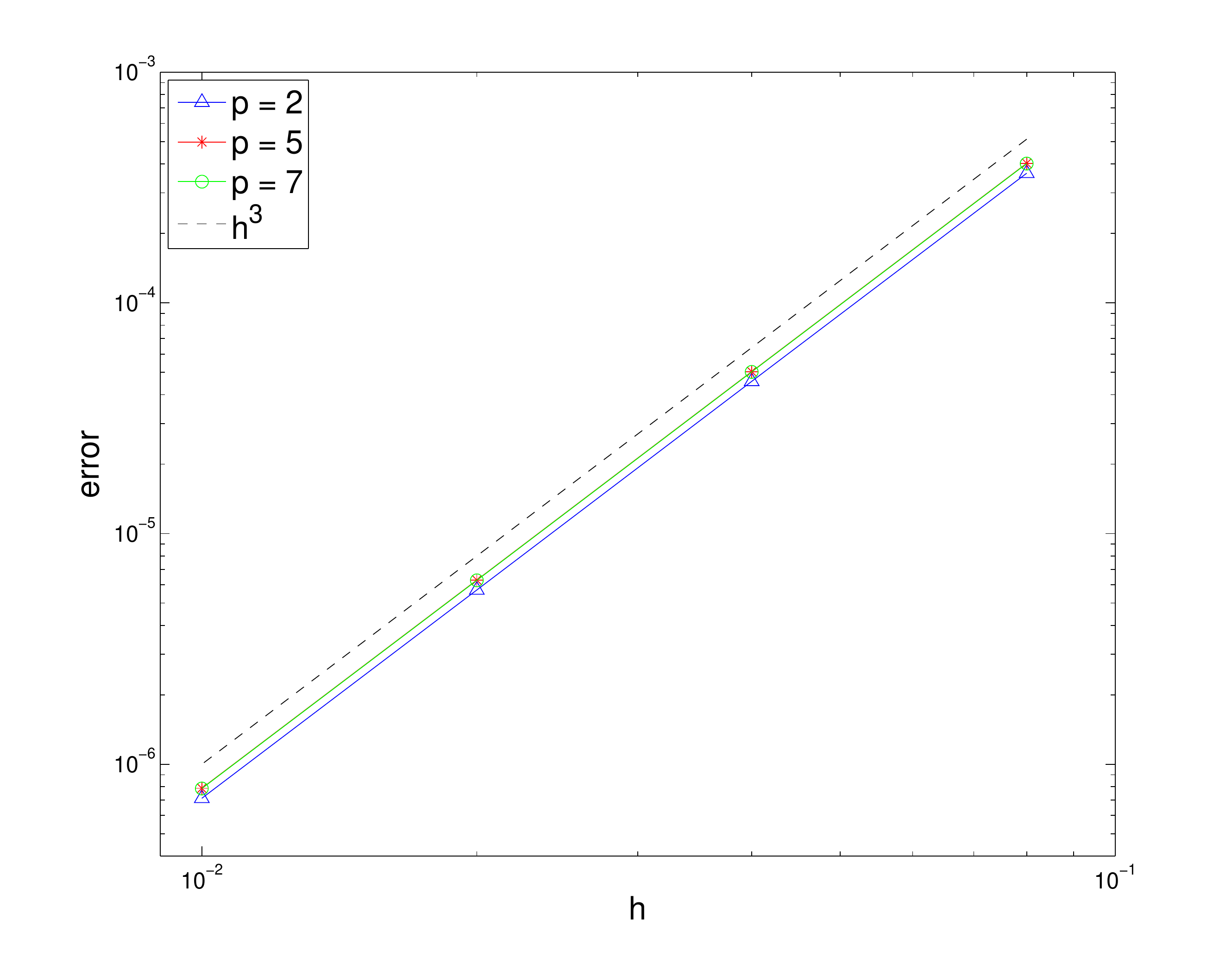}  \\
{\footnotesize (a) continuous FE ($h$ = 0.08, 0.04, 0.02, 0.01)} & (b) {\footnotesize IP-DG ($h$ = 0.08, 0.04, 0.02, 0.01)}
\end{tabular}
\begin{tabular}{c}
\includegraphics[scale=0.2]{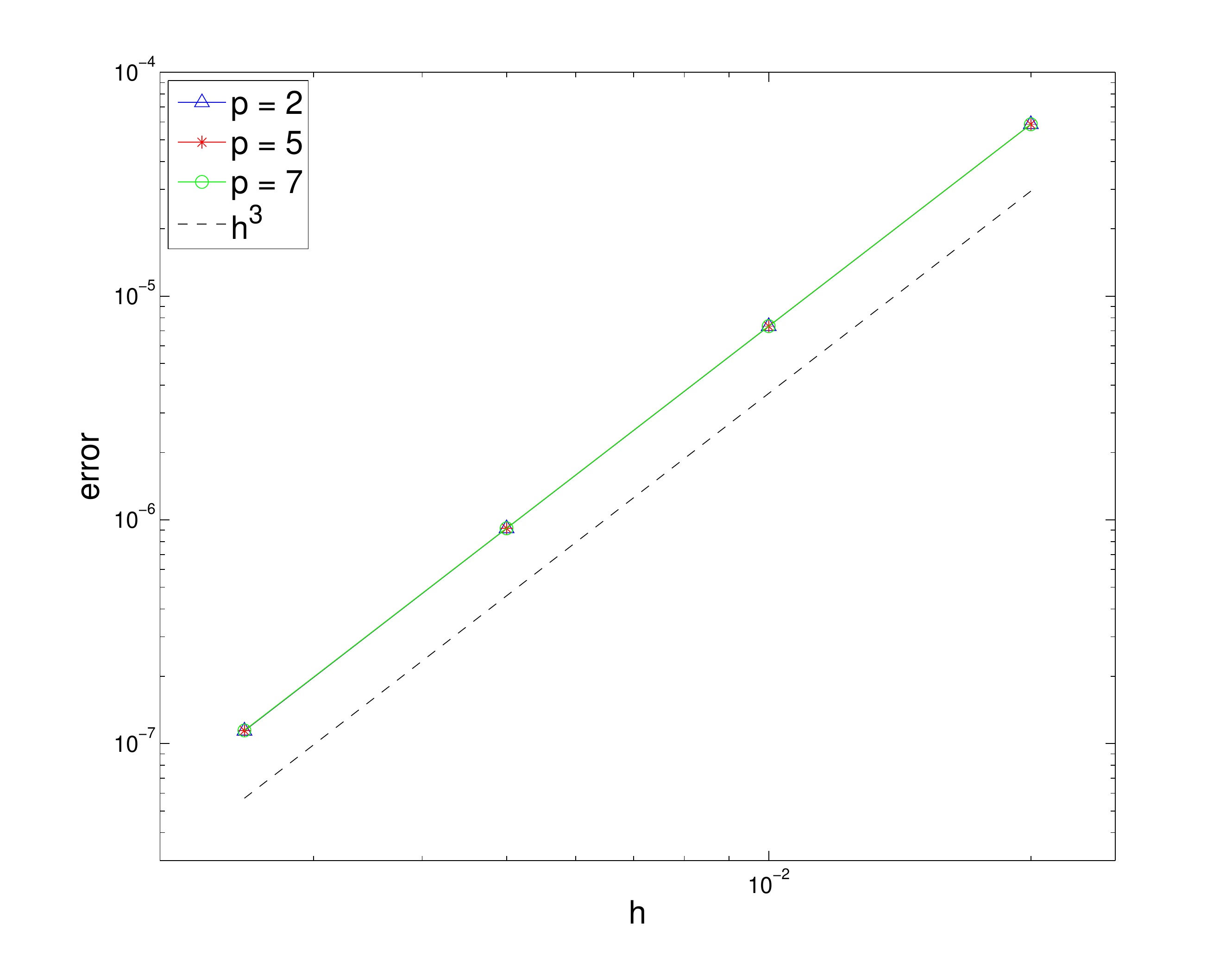} \\ {\footnotesize (c) nodal DG ($h$ = 0.02, 0.01, 0.005, 0.0025)}
\end{tabular}
\caption{LTS-AB$3$($p$) error vs.\ $h = h^{\mbox{\scriptsize coarse}}$ for ${\cal P}^2$ finite elements with $p=2, 5, 7$.} \label{1d_err_P2}
\end{figure}

\begin{figure}[t!]
\centering
\begin{tabular}{cc}
\includegraphics[scale=0.2]{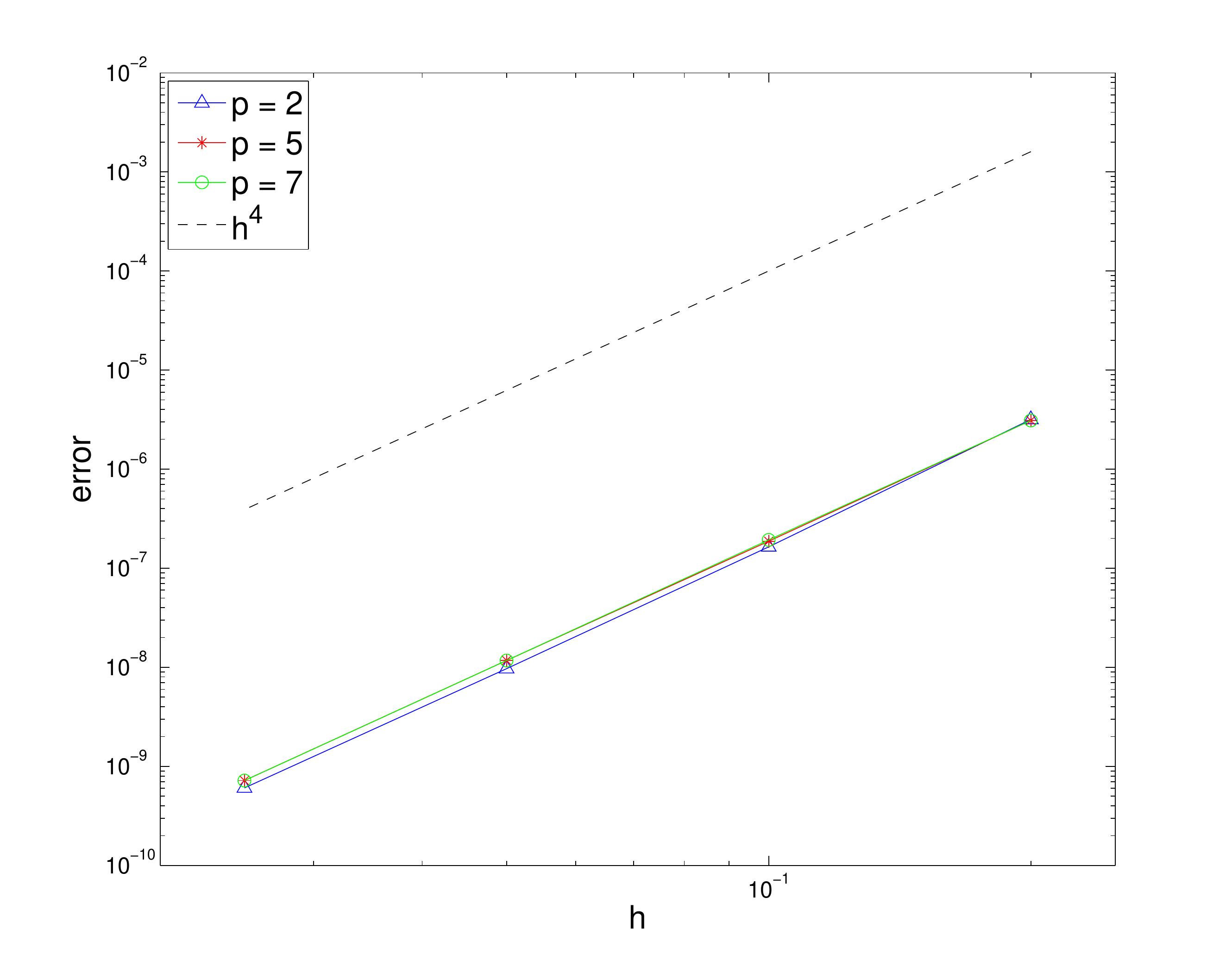}  & 
\includegraphics[scale=0.2]{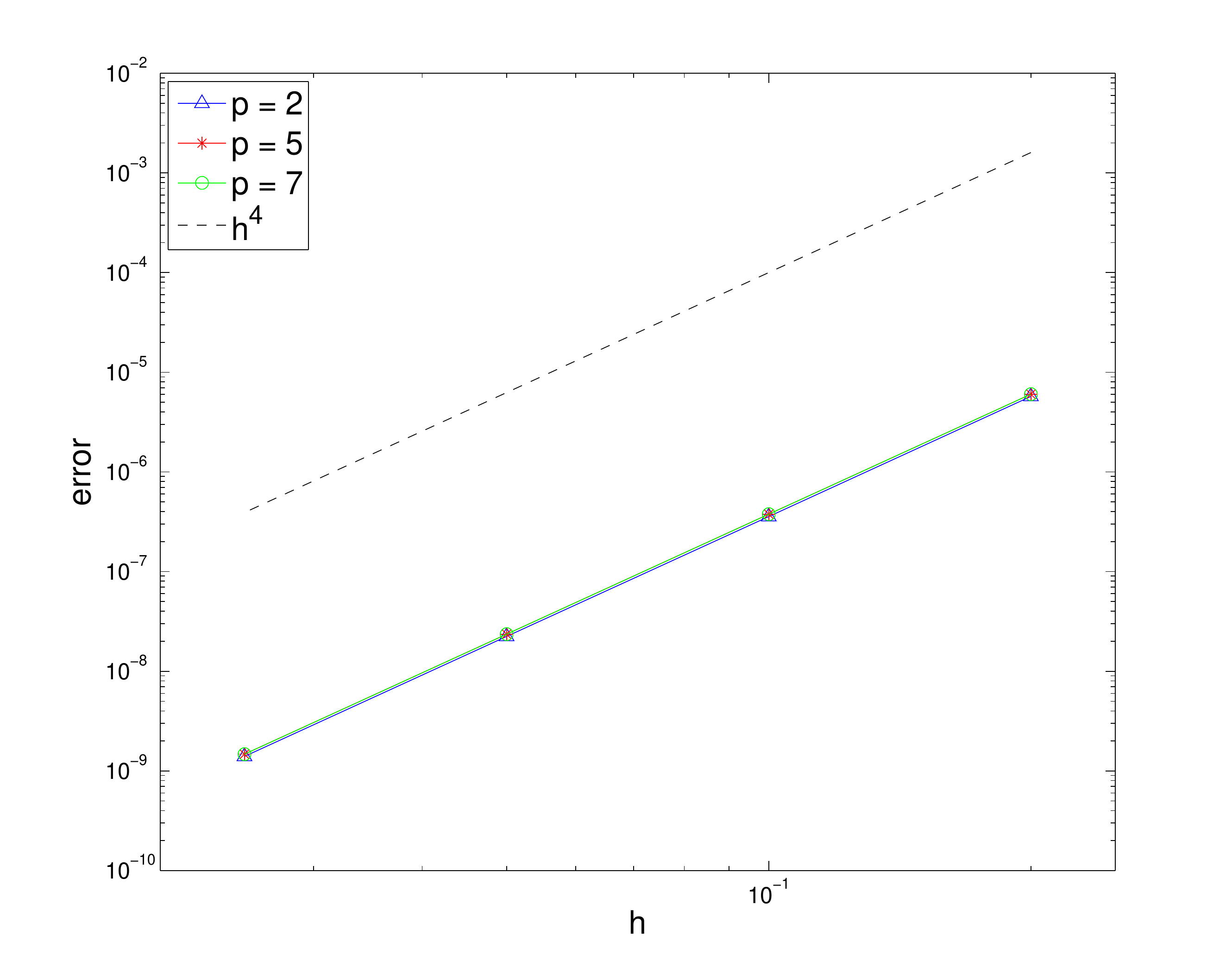}  \\
{\footnotesize (a) continuous FE ($h$ = 0.2, 0.1, 0.05, 0.025)} & {\footnotesize (b) IP-DG ($h$ = 0.2, 0.1, 0.05, 0.025)}
\end{tabular}
\begin{tabular}{c}
\includegraphics[scale=0.2]{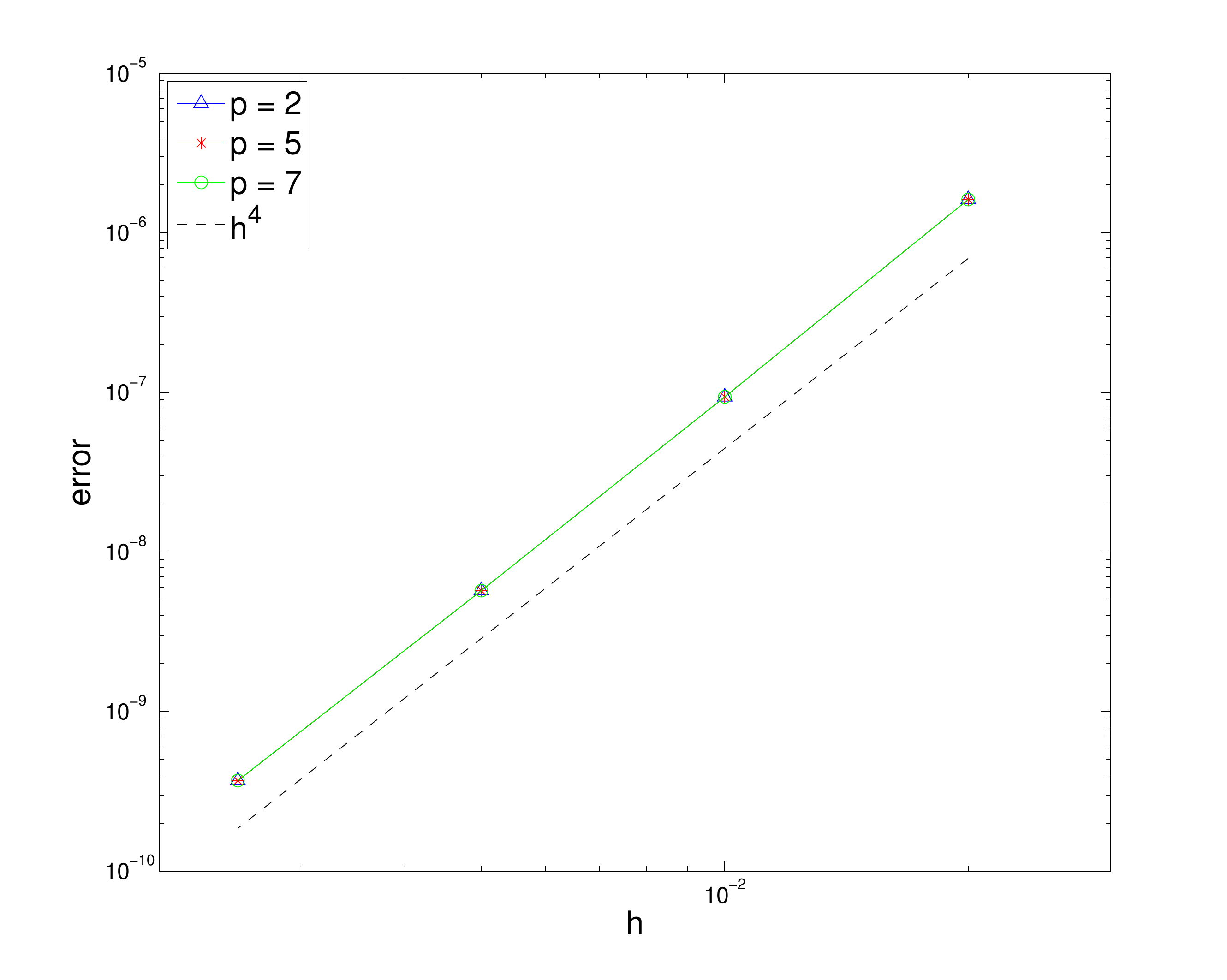} \\ {\footnotesize (c) nodal DG ($h$ = 0.02, 0.01, 0.005, 0.0025)}
\end{tabular}
\caption{LTS-AB$4$($p$) error vs.\ $h = h^{\mbox{\scriptsize coarse}}$ for ${\cal P}^3$ finite elements with $p=2, 5, 7$.} \label{1d_err_P3}
\end{figure}

\begin{remark}
We obtain similar convergence results for other values of $p$ and $\sigma$. In summary, we observe the expected convergence of order $k$ for the 
LTS-AB$k$($p$) schemes, regardless of the spatial FE discretization and independently of the number of local time-steps $p$ and the damping coefficient $\sigma$.
\end{remark}

\subsection{Two-dimensional example}
To illustrate the usefulness of the LTS approach, we consider (\ref{model_eq_1}) in the computational domain $\Omega$, 
shown in Fig.~\ref{mesh_antenna}: it corresponds to an initially rectangular domain of size 
$[0, 3] \times [0, 1]$ from which the shape of a roof mounted antenna of thickness 0.01 has ben cut out. We set the 
constant wave speed $c = 1$ and the constant damping coefficient $\sigma = 0.1$.  On the boundary of $\Omega$, we 
impose homogeneous Neumann instead of Dirichlet conditions and choose as initial conditions the vertical Gaussian plane wave 
\begin{equation*} \begin{split}
u_0(x, y) & = \exp \left ( -(x-x_0)^2 / \delta^2 \right ) \,, \quad \mbox{for all} \ (x,y) \in \Omega\,, \\
v_0(x, y) & = 0 \,, \quad \mbox{for all} \ (x,y) \in \Omega\,,
\end{split} \end{equation*}
centered about $x_0 = 1.3$ and of width $\delta=0.01$.

\begin{figure}[t]
\begin{center}
\includegraphics[width=.55\textwidth]{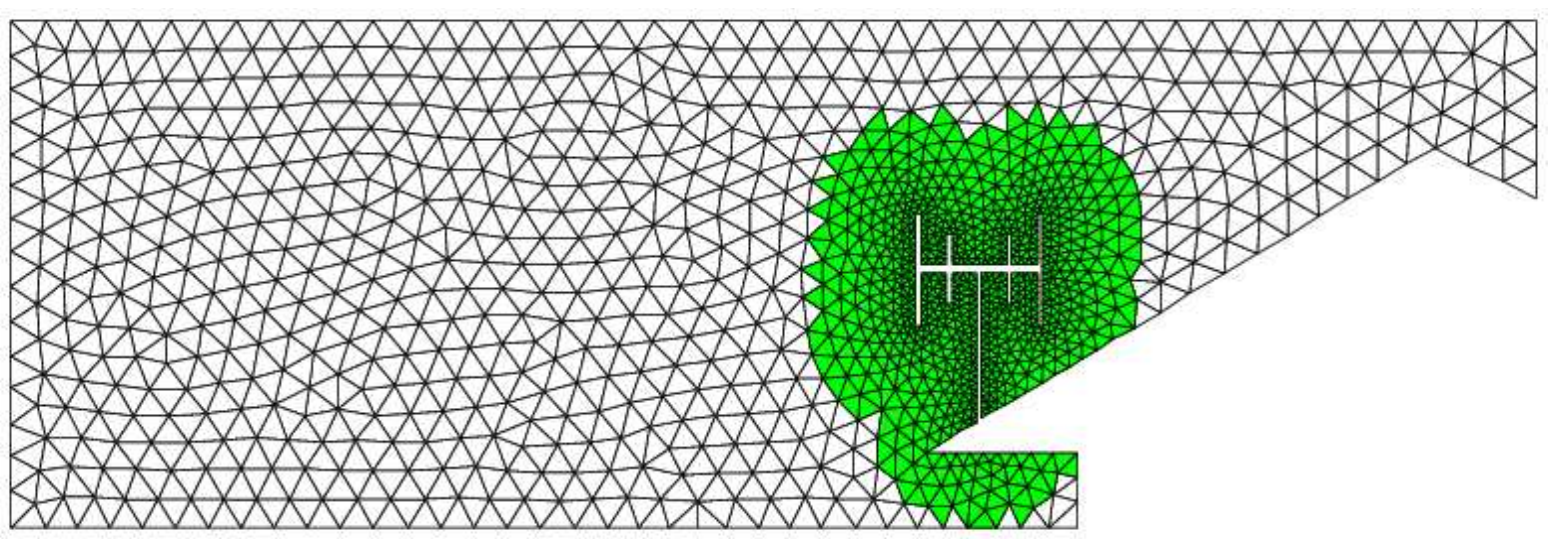} \hspace{1cm}
\includegraphics[scale = 0.25]{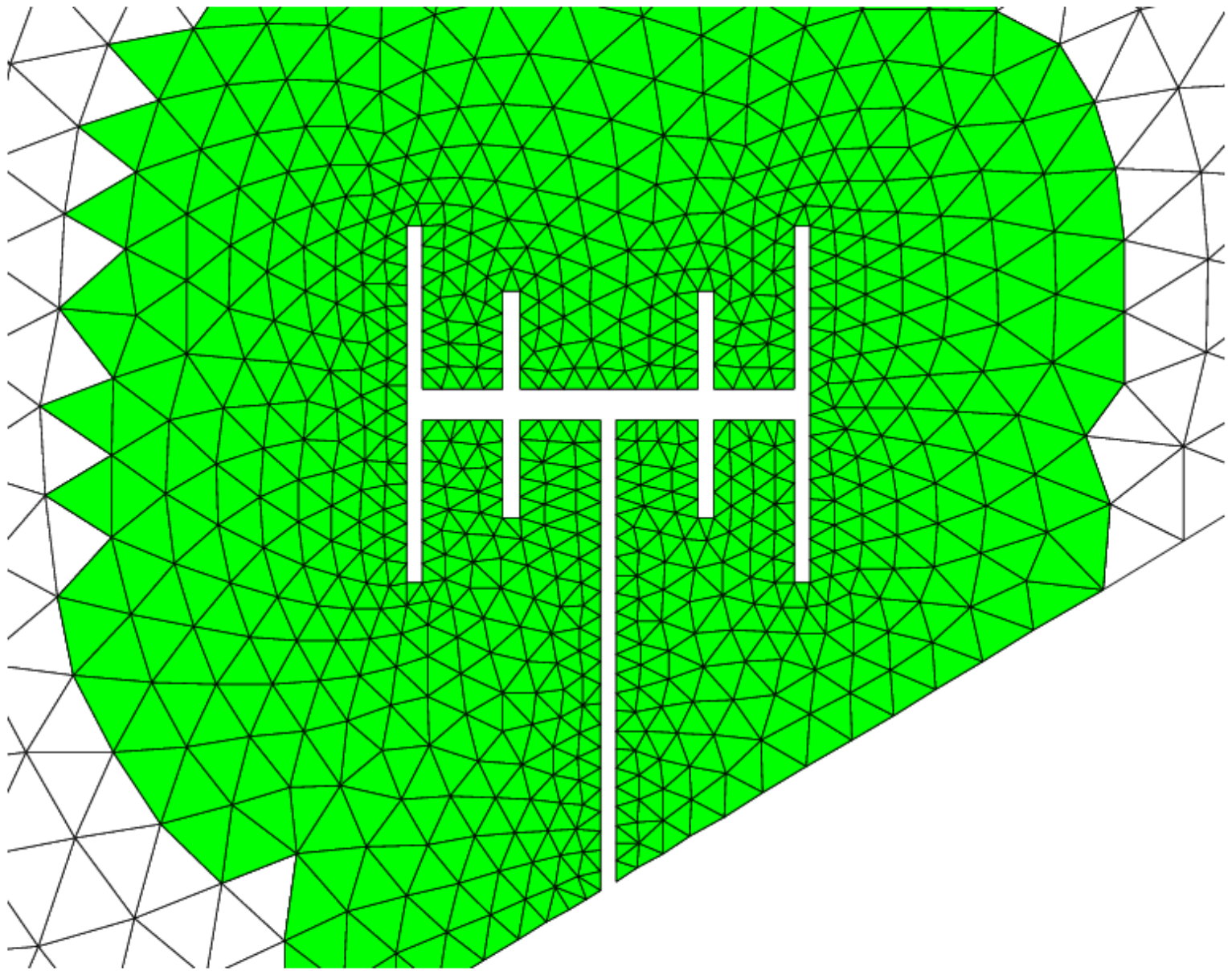}
\end{center}
\caption{The initial triangular mesh of the computational domain $\Omega$ (left); zoom on the ``fine'' mesh indicated by the darker (green) triangles (right).}
\label{mesh_antenna}
\end{figure}

For the spatial discretization we opt for ${\cal P}^2$ continuous finite elements with mass lumping \cite{CJRT01, Coh02}. First, 
$\Omega$ is discretized with triangles of maximal size $h^{\mbox{\scriptsize coarse}}~=~0.05$. However, such coarse triangles 
do not resolve the small geometric features of the antenna, which require 
$h^{\mbox{\scriptsize fine}} \approx h^{\mbox{\scriptsize coarse}} / 7$, as shown in Fig.~\ref{mesh_antenna}. 
Then, we successively refine the entire mesh three times, each time splitting every triangle into four. 
Since the initial mesh in $\Omega$ is unstructured, the boundary between the fine and the coarse mesh is not well-defined. 
Here we let the fine mesh correspond to all triangles with $h < 0.6 \, h^{\mbox{\scriptsize coarse}}$ in size, i.e. the darker (green) 
triangles in Fig.~\ref{mesh_antenna}. The corresponding degrees of freedom in the finite element solution are then selected merely by setting to one the 
corresponding diagonal entries of the matrix ${\mathbf P}$ (see Section 3.2).

For the time discretization, we choose the third-order LTS-AB$3$($7$) time-stepping scheme with $p = 7$, which for every time-step $\Delta t$ takes seven local 
time-steps inside the refined region that surrounds the antenna (see Fig.~\ref{mesh_antenna}). Thus, the numerical method is third-order accurate in both space 
and time under the CFL condition $\Delta t = 0.07 \, h^{\mbox{\scriptsize coarse}}$, determined experimentally. If instead the same (global) time-step 
$\Delta t$ was used everywhere inside $\Omega$, it would have to be about seven times smaller than necessary in most of $\Omega$. As a starting procedure, 
we employ a standard fourth-order Runge-Kutta scheme.

In Fig.~\ref{antenna_sol}, snapshots of the numerical solution are shown at different times. The vertical Gaussian pulse initiates two plane waves propagating 
in opposite directions. At $t=0.5$, the right-moving wave impinges first on the antenna and than on the tip of roof. Multiple reflections occur as the lower part 
of the wave bounces back. Meanwhile, the upper part of the plane wave has proceeded to the right without any spurious reflection between the coarse and the refined 
regions.

\begin{figure}[t]
\begin{center}
\includegraphics[width=.45\textwidth]{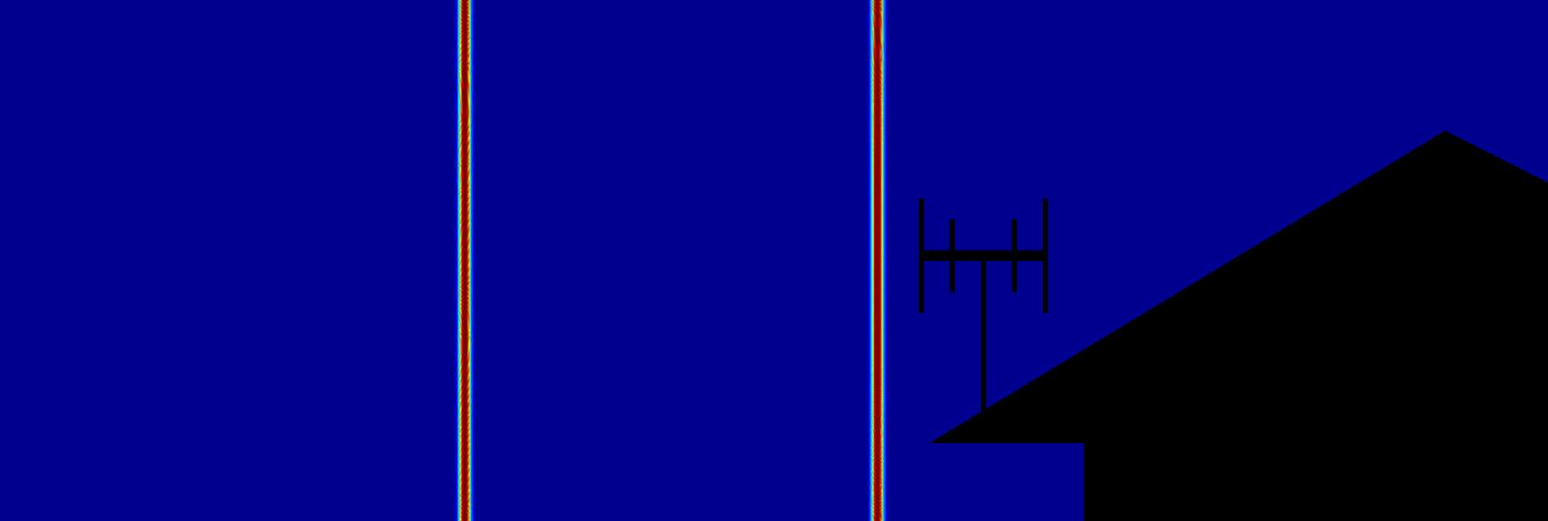} \hspace{1cm}
\includegraphics[width=.45\textwidth]{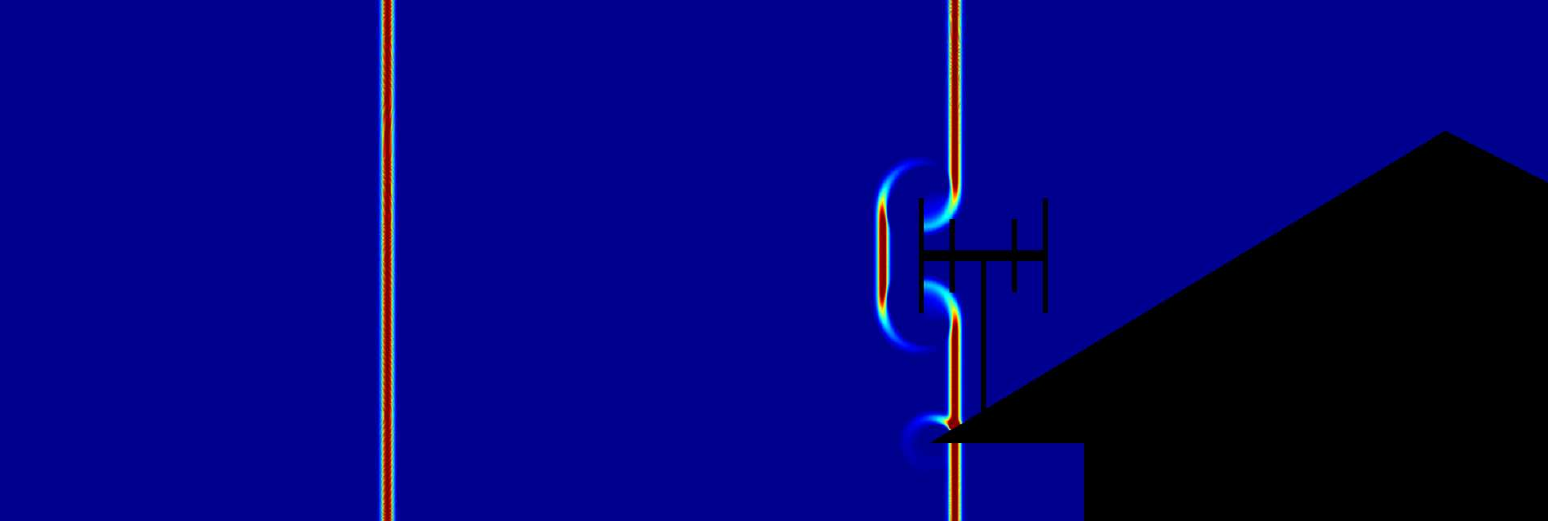}

\bigskip

\includegraphics[width=.45\textwidth]{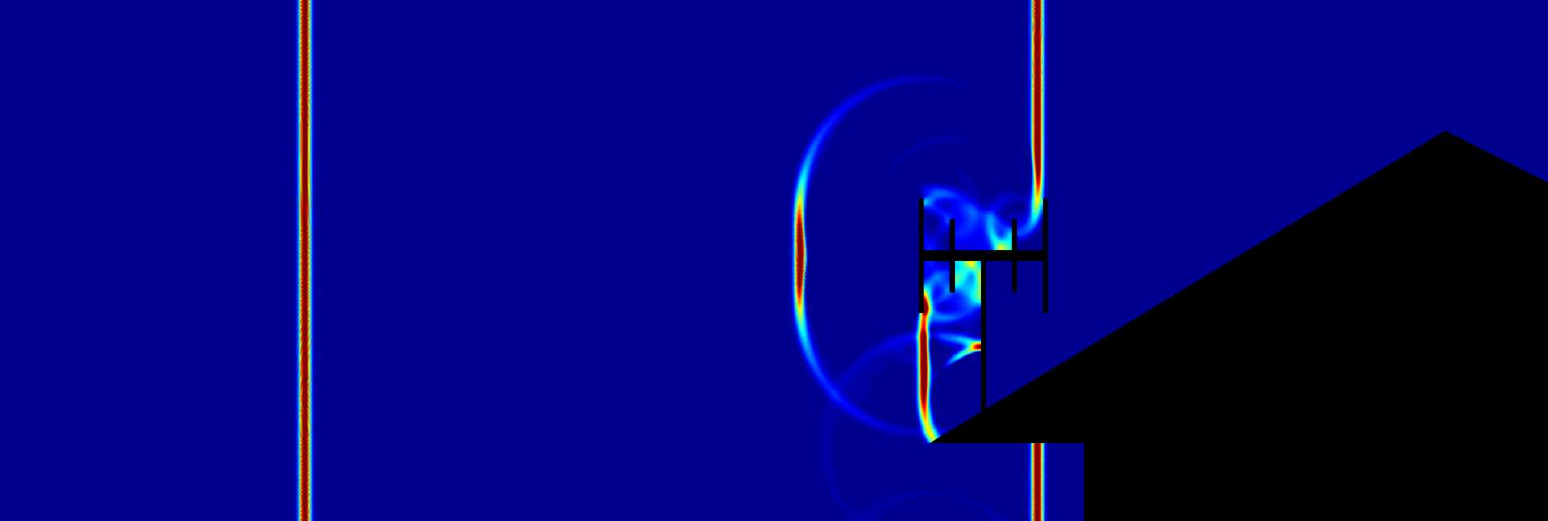} \hspace{1cm}
\includegraphics[width=.45\textwidth]{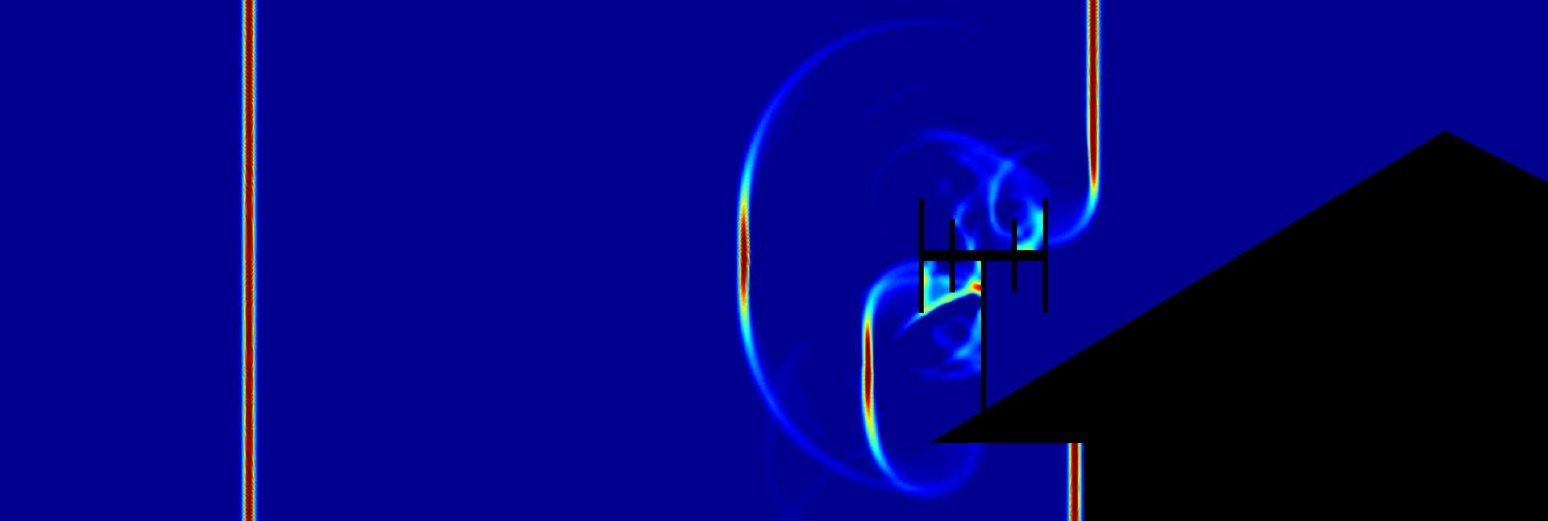}

\bigskip

\includegraphics[width=.45\textwidth]{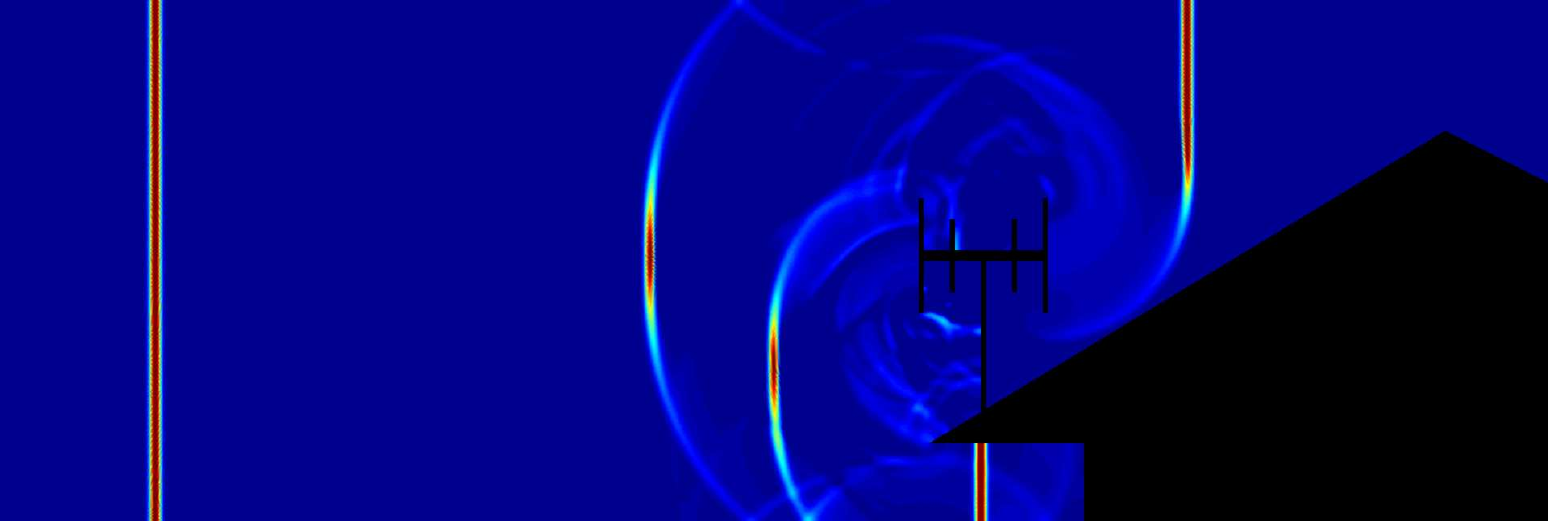} \hspace{1cm}
\includegraphics[width=.45\textwidth]{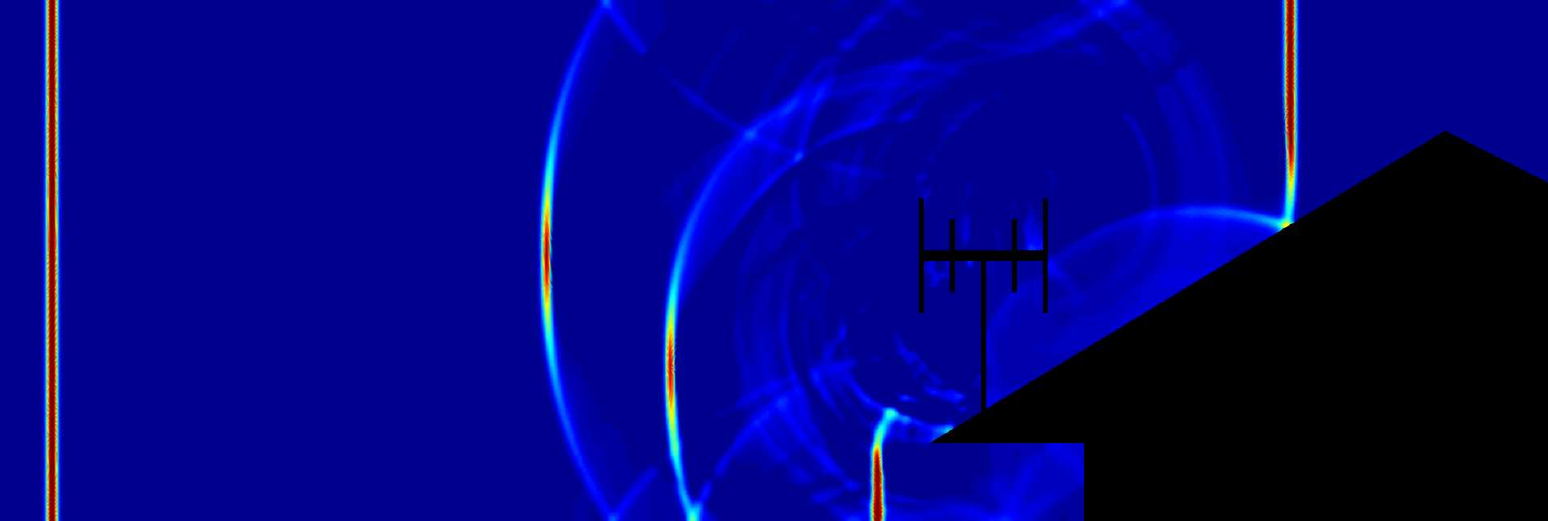}
\end{center}
\caption{Two-dimensional example: the solution is shown at times $t=$0.4, 0.55, 0.71, 0.82, 1 and 1.2.} \label{antenna_sol}
\end{figure}

\section{Conclusion}
Starting from classical Adams-Bashforth methods, we have presented explicit local time-stepping (LTS) schemes 
for damped wave equations, which permit arbitrarily small time-steps precisely where the smallest elements 
in the mesh are located. Thus, when combined with a
spatial finite element discretization with an essentially diagonal mass matrix, 
the resulting time-marching schemes are fully explicit. 
The general $k$-th order LTS scheme, denoted by LTS-AB$k$, is given by the LTS-AB$k$($p)$ Algorithm 
in Section 3.2. As shown in Section 3.4, it is 
$k$-th order accurate in time. Thus when combined with a spatial finite element discretization of order 
$k-1$, the resulting numerical scheme yields optimal $k$-th order space-time convergence in the $L^2$-norm. 
The derivation of the LTS-AB$k$($p)$ algorithm immediately generalizes to
the inhomogeneous case with nonzero external forcing. 

The LTS-AB$k$($p)$ scheme has been combined with three distinct finite element discretizations: 
standard $H^1$-conforming finite elements, 
an IP-DG formulation, and nodal DG finite elements.
In all cases, our numerical results demonstrate that the resulting LTS-AB$k$ schemes of order $k \geq 3$ 
have optimal CFL stability properties regardless of the mesh size $h$, the global to local step-size ratio $p$,
or the dissipation $\sigma$. For $k=2$, however, the CFL condition is sub-optimal as the maximal 
time-step allowed by the LTS-AB2 scheme is only about 80\% of that permitted by the standard second-order
Adams-Bashforth scheme on an equidistant mesh.
In contrast to the energy conserving LTS methods that were developed for the (undamped) wave equation in
\cite{DG09}, the LTS-AB$k$ methods with $k\geq 3$ always achieve the optimal CFL condition without 
overlap of the fine and the coarse region. In fact, they even do so in the undamped regime; hence, the
LTS-AB$k$ methods with $k\geq 3$ can even be used when $\sigma$ vanishes locally or throughout $\Omega$,
yet even in the absence of damping and forcing, they will not conserve the energy.

Since the LTS methods presented here are truly explicit, their parallel implementation is straightforward. Let $\Delta t$ denote the 
time-step imposed by the CFL condition in the coarser part of the mesh. Then, during every (global) time-step $\Delta t$, each local time-step of size 
$\Delta t / p$ inside the fine region of the mesh, simply corresponds to sparse matrix-vector multiplications that only involve the degrees of freedom
associated with the fine region of the mesh. Those ``fine'' degrees of freedom can be selected individually and without any restriction by setting the 
corresponding entries in the diagonal projection matrix $P$ to one; in particular, no adjacency or coherence in the numbering of the degrees of freedom is assumed. 
Hence the implementation is straightforward and requires no special data structures.  

In the presence of multi-level mesh refinement, each local time-step in the fine
region can itself include further local time-steps inside a smaller subregion with an even higher degree of local mesh refinement.
The explicit local time-stepping schemes developed here for the scalar damped wave 
equation immediately apply to other 
damped wave equations, such as in electromagnetics 
or elasticity; in fact, they can be used for general linear first-order hyperbolic systems.

\section*{Acknowledgements}
We thank Manuela Utzinger and Max Rietmann for their help with the numerical experiments.


\end{document}